\documentclass[11pt]{article}
\usepackage{latexsym,amsmath,url,epsfig}
\usepackage{textcomp}
\usepackage{multirow}
\usepackage{amsfonts,euscript}
\usepackage{amsmath}
\usepackage{amssymb}
\usepackage{amsfonts}
\usepackage{amsthm}
\usepackage{mathrsfs}
\usepackage[all]{xy}
\usepackage{epstopdf}
\usepackage{rotating}
\usepackage{appendix}
\usepackage{multirow,rotating}
\usepackage{url}
\usepackage{algorithm}
\usepackage{algorithmicx}
\usepackage{algpseudocode}
\usepackage{setspace}
\usepackage{listings}
\usepackage{xcolor}
\usepackage[normalem]{ulem}
\usepackage[showdeletions]{color-edits}
\usepackage{tikz}
\usepackage{graphics}
\usepackage[breaklinks=true]{hyperref}
\usepackage[english]{babel}
\usepackage{datetime}
\usepackage{booktabs}
\usepackage{rotfloat}
\usepackage{framed}
\usepackage[authoryear]{natbib}
\usepackage{graphicx}
\usepackage{caption}
\usepackage{subcaption} 

\usetikzlibrary{chains,shapes,decorations,arrows,calc,arrows.meta,fit,positioning}
\tikzset{three sided/.style={
draw=none,
append after command={
[shorten <= -0.5\pgflinewidth]
([shift={(-0.5\pgflinewidth,-0.5\pgflinewidth)}]\tikzlastnode.north east)
edge([shift={( 0.5\pgflinewidth,-0.5\pgflinewidth)}]\tikzlastnode.north west) 
([shift={( 0.5\pgflinewidth,-0.5\pgflinewidth)}]\tikzlastnode.north east)
edge([shift={( 0.5\pgflinewidth,-.95\pgflinewidth)}]\tikzlastnode.south east)  
([shift={( -0.5\pgflinewidth,-0.5\pgflinewidth)}]\tikzlastnode.south west)
edge([shift={(0.5\pgflinewidth,-0.5\pgflinewidth)}]\tikzlastnode.south east)
}
}
}
\tikzset{three sided2/.style={
draw=none,
append after command={
[shorten <= -0.5\pgflinewidth]
([shift={( -0.5\pgflinewidth,-0.5\pgflinewidth)}]\tikzlastnode.north west)
edge([shift={( -0.5\pgflinewidth,-.95\pgflinewidth)}]\tikzlastnode.south west)  
([shift={( 0.5\pgflinewidth,-0.5\pgflinewidth)}]\tikzlastnode.north east)
edge([shift={( 0.5\pgflinewidth,-.95\pgflinewidth)}]\tikzlastnode.south east)  
([shift={( -0.5\pgflinewidth,-0.5\pgflinewidth)}]\tikzlastnode.south west)
edge([shift={(0.5\pgflinewidth,-0.5\pgflinewidth)}]\tikzlastnode.south east)
}
}
}
\tikzset{three sided3/.style={
draw=none,
append after command={
[shorten <= -0.5\pgflinewidth]
([shift={(-0.5\pgflinewidth,-0.5\pgflinewidth)}]\tikzlastnode.north east)
edge([shift={( 0.5\pgflinewidth,-0.5\pgflinewidth)}]\tikzlastnode.north west) 
([shift={( -0.5\pgflinewidth,-0.5\pgflinewidth)}]\tikzlastnode.north west)
edge([shift={( -0.5\pgflinewidth,-.95\pgflinewidth)}]\tikzlastnode.south west)  
([shift={( 0.5\pgflinewidth,-0.5\pgflinewidth)}]\tikzlastnode.north east)
edge([shift={( 0.5\pgflinewidth,-.95\pgflinewidth)}]\tikzlastnode.south east)  
}
}
}

\definecolor{revisionorange}{rgb}{1.0,0.45,0.0}

\addauthor[Jin]{jg}{revisionorange}
\graphicspath{{pic/}}
\newtheorem{lemma}{Lemma}

\theoremstyle{definition}

\newcommand{\abs}[1]{\lvert#1\rvert}

\algdef{SE}[DOWHILE]{Do}{doWhile}{\algorithmicdo}[1]{\algorithmicwhile\ #1}

\allowdisplaybreaks

\begin{document}

\title{Diffusion-Based Policies for Dynamic Control of Stochastic Processing Networks}

\author{
Bar{\i}\c{s} Ata\thanks{Booth School of Business, University of Chicago,
\texttt{baris.ata@chicagobooth.edu}}
\and
Jin Guang\thanks{Booth School of Business, University of Chicago,
\texttt{jin.guang@chicagobooth.edu}}
\and
J. Michael Harrison\thanks{Graduate School of Business, Stanford University,
\texttt{mike.harrison@stanford.edu}}
\and
Nian Si\thanks{Dept of IEDA, Hong Kong University of Science and Technology,
\texttt{niansi@ust.hk}\\ Our code is available at \url{https://github.com/J-Guang/SPN-dynamic-control}}
}

\date{}
\maketitle
\begin{abstract}
We consider a processing network model with $m$ job classes or buffers, exogenous input flows into some classes, $n$ processing activities, and $p$ servers. Each activity is either a specified server processing jobs of a specified class, or a fictional input server delivering jobs of a specified class; jobs change class in Markovian fashion after completing service. A standard multiclass queueing network, with its one-to-one correspondence between job classes and activities, is a special case, but our general model allows two or more ways to process a given class, and some or all input flows may be turned away at the system manager's discretion.

Costs are linear: a holding cost per time unit for each class $i$ job in the system, and a rejection penalty for each class $i$ arrival denied access $(i=1,\ldots,m)$. The system manager makes input control, job routing, and order-of-service decisions to minimize expected discounted costs over an infinite horizon.

We formulate an approximating Brownian control problem (BCP) whose state space is the $m$-dimensional nonnegative orthant; control is a drift vector chosen from a bounded polyhedral set, based on dynamic state observations. Using recently developed computational methods, the BCP can be solved numerically in dimensions up to at least $m=50$, and we explain how the numerical solution is translated into an implementable control policy for the queueing system of original interest.

Previous work on heavy traffic diffusion approximations suggests that this policy is nearly optimal in the heavy traffic parameter regime, and our numerical examples support that conjecture. We also discuss its advantage over an alternative approach, featured in our previous work, where the BCP is replaced by a lower-dimensional "equivalent workload formulation" that is computationally efficient but difficult to interpret in the network of original interest.
\end{abstract}

\section{Introduction}\label{sec:introduction}

This paper continues the work reported previously in \citet{ata2023drift} and \citet{ata2024singular}. The first of those papers developed a computational method for the solution of certain stochastic control problems, and the second one adapted that method for the solution of a different but related class of problems. To be specific, the first paper deals with drift control of reflected Brownian motion (RBM), while the second one deals with singular control of RBM. 

The motivation for that earlier work, and for the research reported here, comes from applications in queueing theory. We continue an established tradition of diffusion approximations, also called Brownian approximations, for queueing network models. Following recent usage in the stochastic modeling literature, we use the term ``stochastic processing network" (SPN) to identify a broad class of stochastic system models, and reserve the term ``queueing network" for models within that class having a particular special structure (see Section \ref{sec:general}).

Our focus is on Brownian approximations for SPN control problems. The associated literature began with the analysis of a specific example (the ``criss-cross example" introduced in Section \ref{sec:general} of this paper) by \cite{harrison1989scheduling} and \cite{martins1996heavy}, and with a general framework developed by \cite{harrison1988brownian}. Our goal is to derive approximately optimal solutions for such network control problems in the ``heavy traffic" parameter regime. Both of our earlier papers on computational methods featured queueing-theoretic examples, including Brownian approximations for SPN control problems that had been derived and discussed in earlier work. However, neither paper provided a systematic treatment of the process by which the computed solution for an approximating Brownian control problem is translated into an implementable policy for the SPN of original interest.  

The purpose of this paper is to provide such a systematic treatment. We begin by describing a general network model (see Section \ref{sec:general}), after which a Brownian approximation is developed for the associated network control problem (Sections \ref{sec:Centering} through \ref{sec:drift1}). The state space of the Brownian approximation is the $m$-dimensional nonnegative orthant, where $m$ is the number of job classes in the SPN. In Section \ref{sec:policy} we lay out a general method for translating the computed solution of the Brownian control problem into an implementable policy for control of the original SPN. The method is illustrated using numerical examples in Section \ref{sec:Tests}, and those computational results are discussed further in Section \ref{sec:Discussion}. There are also several appendixes that provide background calculations and computational details.
 
In previous work on Brownian approximations for SPN control, the notion of an \textit{equivalent workload formulation} (EWF) has played a prominent role. That concept was developed in its most general form by \cite{harrison1997dynamic}, who observed the following: the natural diffusion approximation for a network control problem with $m$ job classes is a Brownian control problem (BCP) in the $m$-dimensional nonnegative orthant, but in most cases that initial BCP is logically equivalent to another one (the EWF) whose dimension $d$ is smaller than $m$, often much smaller. 

By focusing on the $d$-dimensional EWF, rather than the initial BCP, one can reduce computational effort, often quite substantially, but there is a catch: no general method has been proposed to translate the EWF solution into an implementable control policy for the network of original interest, and the difficulty of that translation increases with the dimension reduction effected by the EWF. That is, greater computational gain is typically accompanied by greater difficulty implementing the computed solution. See Section \ref{sec:ThreeStation} for discussion of an example in which that difficulty arises.

\section{Literature Review}\label{sec:literature}

Our work builds on five interrelated streams of literature,
which are reviewed in turn below. These streams cover a broad range of topics, from model formulation at one extreme to computational methods at the other.

\textbf{Dynamic control of stochastic processing networks in heavy
traffic.} The use of
Brownian approximations for purposes of control, as opposed to descriptive performance analysis, began with
\citet{harrison1988brownian}. That paper formulated a general Brownian
control problem (BCP) that approximates a multiclass queueing network
with dynamic scheduling capability; we follow that formulation closely in Sections \ref{sec:general} through \ref{sec:BrownianAnalog} of this paper. See also \citet{harrison2000brownian} and \citet{harrison2003broader} for further generalizations. Early applications of Harrison's Brownian framework include the following: the
criss-cross example studied by \citet{harrison1989scheduling}, which serves
as our first example in this paper; the closed-network analysis by
\citet{harrison1990scheduling}; and the analyses of
\citet{wein1991brownian, wein1992scheduling}, which featured discretionary
routing and controllable inputs, respectively. \citet{laws1992resource} and \citet{kellylaws1993dynamic} also studied
Brownian approximations of networks with a dynamic routing capability,
identifying the phenomenon of resource pooling that figures prominently
in later work.

\textbf{Equivalent workload formulations.} As recounted in Sections \ref{sec:introduction} and \ref{sec:EWF} below,
\citet{harrison1997dynamic} developed in general form the notion of an
equivalent workload formulation (EWF), whereby an $m$-dimensional
initial BCP is reduced to an equivalent singular control
problem of lower dimension $d$.  \citet{harrison2000brownian} identified a
canonical representation of the workload matrix $M$ (see also the
correction, \citealp{harrison2006correction}), and extended that development to the more
general class of stochastic processing networks considered in this paper. 

An important special case arises when the EWF is
one-dimensional, a situation described by the phrase ``complete
resource pooling'': the many processing resources of the network then
merge, in the heavy traffic limit, into a single super-resource whose
workload summarizes system status; see \citet{laws1992resource},
\citet{kellylaws1993dynamic} and \citet{harrison-lopez1999} for early
treatments and \citet{ata-kumar2005} for a general formulation. For
parallel-server systems, \citet{pesic2016dynamic} extended the EWF
machinery, identifying an associated graphical structure and decoupled workload
matrix; their three-server example is the second of the test problems
treated in Sections \ref{sec:ThreeExamples} and \ref{sec:Tests} of this paper. As
emphasized in Sections \ref{sec:introduction} and \ref{sec:Discussion}, the EWF reduces
computational effort, often very substantially, but no general method
has been proposed for translating an EWF solution into an implementable
control policy, and that difficulty motivates the alternative path
pursued in this paper.

\textbf{From Brownian solutions to implementable policies.} Broadly speaking, the policies proposed to date fall into two
categories, namely, continuous-review and discrete-review. Under a continuous-review policy, the system manager makes decisions continually as the system state evolves, 
typically according to a
dynamic priority or threshold rule distilled from the Brownian
solution. For the criss-cross network, \citet{harrison1989scheduling}
proposed a threshold policy of this kind, and \citet{martins1996heavy}
proved its asymptotic optimality. Subsequent work by
\citet{budhiraja2005large}, \citet{budhiraja2008optimal} and
\citet{budhiraja2017construction} extended the rigorous theory to
further parameter regimes of the criss-cross example, in the last case
constructing an asymptotically optimal policy directly from a free
boundary problem. For parallel-server systems,
\citet{bellwilliams2001dynamic, bellwilliams2005dynamic} proved
asymptotic optimality of continuous-review threshold policies derived
from the one-dimensional EWF, and \citet{mandelbaumstolyar2004scheduling}
established the remarkable fact that the generalized $c\mu$ rule, which
requires neither thresholds nor advance knowledge of demand rates, is
asymptotically optimal for convex holding costs. \citet{maglaras2003continuous} proposed a general family of continuous-review tracking policies. A common feature of the
threshold-type policies in this stream is the presence of safety-stock
parameters, whose values must be specified with care. In practice, one
often chooses them by simulation experimentation to obtain good performance; the
discussion of the criss-cross example in Section \ref{sec:ThreeExamples}
illustrates the role of such parameters.

Discrete-review policies, in contrast, review system status at discrete
points in time, planning activity levels over the next review period by
solving a linear program whose data derive from the Brownian
or fluid analysis. \citet{harrison1996bigstep} proposed the general
BIGSTEP framework, and \citet{harrison1998discrete-review} used it to
prove asymptotic optimality of a discrete-review policy for a
parallel-server system. \citet{ata-kumar2005} proved asymptotic
optimality of discrete-review policies for open processing networks
satisfying the complete resource pooling condition, allowing the full
generality of alternative processing modes featured in our model
formulation; also see \citet{ata-olsen2013}.

The policy proposed in Section \ref{sec:policy} of this paper belongs to the
continuous-review family, but it differs from the threshold-type
policies reviewed above in two respects. First, it is computed from the
gradient $\nabla V(\cdot)$ of the modified BCP's value function,
evaluated continually at the observed (scaled) state, rather than from
a small number of thresholds distilled in advance; no safety-stock
parameters need to be specified or tuned. Second, because the modified
BCP retains the $m$-dimensional state descriptor of the original
network, its solution prescribes capacity allocations directly, so the
interpretation step is immediate even for networks in which
interpretation of an EWF solution is challenging.

\textbf{Drift-rate control of Brownian system models.} The modified BCP formulated in Section \ref{sec:drift1} is a drift-rate
control problem: the controller continually chooses a drift vector from
a bounded polyhedral action space, the state process being a reflected
diffusion in the nonnegative orthant. One-dimensional problems of this
general type have a long history, beginning with the bounded-velocity
``follower'' problems solved in closed form by \citet{benevs1980some}.
\citet{ataharrisonshepp2005drift} solved a drift-rate control problem
for one-dimensional RBM, motivated by heavy traffic approximations of queueing systems
with adjustable service rates. Subsequent work in this vein includes applications such as dynamic
pricing and lead-time quotation \citep{ccelik2008dynamic}, joint buffer
sizing and rate control with impatient customers
\citep{ghoshweerasinghe2007optimal, ghoshweerasinghe2010optimal}, drift
control with changeover costs motivated by build-to-order manufacturing
\citep{ormecimatoglu2011drift}, volunteer staffing of gleaning
operations \citep{atalee2019dynamic}, and dynamic pricing, outsourcing
and scheduling for a make-to-stock queue \citep{ata2023approximate}.
The problems solved in those papers are all one-dimensional, admitting
analytical or quasi-analytical solutions of the associated Bellman
equations. In contrast, the drift control problems that arise as
modified BCPs in this paper have state dimension $m$, the number of job
classes, which is 8 in our three-station example and can easily reach
several dozen in applications; problems of that scale demand a
different, computational line of attack, to which we now turn.

\textbf{Computational methods for Brownian control problems.} Classical numerical methods for stochastic control are grid-based; see for example, the
Markov chain approximation method of \citet{kushner1991numerical} and the finite-element method of 
\citet{kumar2004numerical}. Such methods are reliable in low dimensions but
suffer from the curse of dimensionality: as a practical matter they are
limited to state dimensions of roughly three to four, which is precisely
why the EWF, with its reduced dimension $d$, has been the preferred
vehicle for computation in the antecedent literature.

A recent and rapidly growing literature uses deep neural networks to
solve high-dimensional partial differential equations and stochastic
control problems, thereby circumventing spatial discretization
altogether. One influential line of work exploits the probabilistic
representation of semilinear parabolic PDEs by backward stochastic
differential equations (BSDEs), a theory initiated by
\citet{pardouxpeng1990adapted} and developed by
\citet{elkaroui1997backward}; see \citet{zhang2017backward} for a
modern treatment. In the ``deep BSDE'' method of
\citet{ehanjentzen2017deep} and \citet{han2018solving}, the unknown
gradient process of the BSDE is parametrized by neural networks, whose
weights are trained by stochastic gradient descent on simulated paths;
see \citet{han-jentzen-e2025} for a survey of this literature and other related developments. In
subsequent applications of this general approach,
\citet{atakasikaralar2025dynamic} solved scheduling problems for
multiclass queues with many servers in dimensions up to 500, and
\citet{ataxu2025dynamic} treated dynamic matching systems.

The computational method employed in this paper is essentially the one developed by \citet{ata2023drift}, who adapted deep BSDE technology to drift control of reflected Brownian motion in an orthant. The novel element of their method is its treatment of boundary reflection terms in the loss function; their method computes both the value function $V(\cdot)$ and its
gradient $\nabla V(\cdot)$, the latter being exactly what is needed for the
policy that we advocate in Section \ref{sec:policy} of this paper. The deep BSDE method of \citet{ata2023drift} is computationally
feasible for state dimensions of 50 or more, and a companion paper
\citep{ata2024singular} treats singular control problems, including the
EWFs referred to above, by approximating them with drift control problems.

\section{General network model with discounted linear costs}
\label{sec:general}

The general model that we consider is one with \(m\) job classes, \(n\) activities, and \(p\) servers. We imagine jobs of each class $i=1,\ldots,m$ as occupying a similarly numbered buffer, regardless of whether the job is being served or is waiting for service. Thus the terms ``class" and ``buffer" can and will be used interchangeably. Each activity \(j = 1, \ldots, n\) consists of an associated server \(s(j) \in \{1, \ldots, p\}\) either processing (that is, serving) jobs of some specified class or else generating arrivals into one or more classes; it is called a \textit{processing activity} in the former case and an \textit{input activity} in the latter case. In the former case, we denote by $b(j) \in \{1, \ldots, m\}$ the class or buffer that is served in activity $j$, and in the latter case we set $b(j)=0$ by convention. For simplicity we assume that input activities create individual arrivals, as opposed to batch arrivals, and similarly, that processing activities serve one job at a time. The treatment to follow can easily be generalized to allow batching, but doing so adds complexity that distracts from the main flow of ideas.

A control policy takes the form of a non-decreasing process \(\{T_j(t), t \geq 0\}\), where \(T_j(t)\) is interpreted as the cumulative amount of capacity that server \(s(j)\) devotes to activity \(j\) over the interval \([0, t]\). As a matter of convention, we ascribe to each server \(k = 1, \ldots, p\) a capacity of 1, so the \(n\)-dimensional process of cumulative capacity allocations \(\{T(t), t \geq 0\}\) must satisfy: 
\begin{equation}
    A \left[ T(t) - T(s) \right] \leq e (t - s) \quad \text{for } 0 \leq s \leq t < \infty, \label{eq:capacity-constraint}
\end{equation}
where \(e\) is the \(p\)-vector of ones, and \(A\) is the \(p \times n\) capacity consumption matrix
\begin{equation}
    A_{kj} =
\begin{cases} 
1 & \text{if } k = s(j), \\
0 & \text{otherwise}.
\end{cases}
\label{eq:defineA}
\end{equation}
Of course, a control policy \(T\) must satisfy other constraints as well; see below for elaboration. For future purposes, it will be convenient to define the $p$-dimensional cumulative idleness process
\begin{equation}
 I(t)=et-AT(t),  \quad t \geq 0,
\label{eq:defineIdleness}
\end{equation}
interpreting $I_k(t)$ as follows for $k=1,\ldots,p$. It is the cumulative idleness of server $k$ over the time interval $[0,t]$, or equivalently, the cumulative amount of server $k$ capacity that goes unused up to time $t$.
\newline

\noindent \textbf{Stochastic primitives and the queue length process.} We take as primitive a set of \(m\)-dimensional, integer-valued stochastic processes \(\{F^j(t), t \geq 0\}\), \(j = 1, \ldots, n\), with the following interpretation: each component process \(F_i^j(\cdot)\) is either monotone increasing or monotone decreasing; in the former case, \(F_i^j(t)\) represents the number of jobs \textit{removed} from buffer \(i\) by the first \( t \) units of capacity that server \(s(j)\) devotes to activity \(j\); and in the latter case, it represents the number of jobs \textit{deposited} in that buffer by the indicated activity. The final primitive stochastic element of our general model is an \(m\)-dimensional exogenous arrival process \(\{E(t), t \geq 0\}\), with the following interpretation: \(E_i(t)\) is the cumulative number of exogenous (that is, uncontrollable) arrivals into buffer \(i\) over the interval \([0,t]\). Given the intended meanings explained in this paragraph, we assume the following \textit{gross structure} for the primitive processes $E(\cdot),F^1(\cdot),\ldots,F^n(\cdot)$: each of them is integer-valued; each component of $E(\cdot)$ is non-decreasing; at most one component of $F^j(\cdot)$ is non-decreasing, while all other components are non-increasing $(j=1,\ldots,n)$; and $E(0)=F^1(0)=\ldots=F^n(0)=0$. These assumptions are satisfied by what \cite{bramson-williams2003} call \textit{unitary networks}; see Definition 3.1 of that paper.

Denoting by \(Q_i(t)\) the content of buffer \(i\) at time \(t\), we define an \(m\)-dimensional queue length process \(\{Q(t), t \geq 0\}\) in the obvious way. Then we have the following basic representation:
\begin{equation}
    Q(t) = Q(0) + E(t) - \sum_{j=1}^n F^j \left( T_j(t) \right) \geq 0, \quad t \geq 0, \label{eq:queueing-dynamics}
\end{equation}
where $Q(0) \in \mathbb{Z}_+^m$ is given as problem data. Two elements of our general model specification are combined in this display: first, the process $Q(\cdot)$ is defined by the equality in \eqref{eq:queueing-dynamics}, and second, the policy $T(\cdot)$ must be chosen so that $Q(\cdot)$ remains nonnegative at all times.

The primitive stochastic processes $E(\cdot),F^1(\cdot),\ldots,F^n(\cdot)$ are mutually independent by assumption, and there are given a vector $\lambda\in \mathbb{R}_+^m$ and an \(m \times n\) matrix \(R\) with columns $R^1,\ldots,R^n$ such that
\begin{equation}
\mathbb{E} \left[ E(t) \right] \sim \lambda t \quad \text{and} \quad \mathbb{E} \left[ F^j(t) \right] \sim R^j t \quad \text{for } j = 1, \ldots, n
\label{eq:first_order}
\end{equation}
as \(t \to \infty\). Thus, \(\lambda_i\) is interpreted as the long-run average exogenous arrival rate into buffer \(i\), and \(R^j_i\) as the long-run average rate at which activity \(j\) removes jobs from buffer \(i\) (if it is positive) or adds jobs to that buffer (if it is negative). Because each activity serves jobs from at most one buffer, \textit{the matrix $R$ has at most one positive element in each column}. Also, to avoid trivial complications, it is assumed that \textit{each row of $R$ contains at least one positive element}, which means that there exists at least one activity that serves jobs of any given class. For future reference, we define the centered processes
\begin{equation}
\hat{E}(t) = E(t) - \lambda t \quad \text{and} \quad \hat{F}^j(t) = F^j(t) - R^j t \quad \text{for } j = 1, \ldots, n \text{   and   } \ \ t \geq 0.
\label{eq:centered_stochastic}
\end{equation}

As a complement to \eqref{eq:first_order} we are given \(m \times m\) covariance matrices \(\Gamma^0, \Gamma^1, \ldots, \Gamma^n \) such that 
\begin{equation}
\text{Cov} \left[ E(t) \right] \sim \Gamma^0 \, t \quad \text{and} \quad \text{Cov} \left[ F^j(t) \right] \sim \Gamma^j \, t \quad \text{for } j = 1, \ldots,n
\label{eq:second_order}
\end{equation}
as \(t \to \infty\). Finally, we assume that $E(\cdot),F^1(\cdot),\ldots,F^n(\cdot)$ each satisfies a functional central limit theorem with first-moment data and second-moment data given by \eqref{eq:first_order} and \eqref{eq:second_order}, respectively. That is, we assume that 
\begin{equation}
\frac{1}{\sqrt r}\hat{E}(r\cdot)\Rightarrow \text{BM}(0,\Gamma^0) \quad \text{and} \quad \frac{1}{\sqrt r}\hat{F}^j(r\cdot)\Rightarrow \text{BM}(0,\Gamma^j) \quad \text{for } j = 1, \ldots,n
\label{eq:FCLT}
\end{equation}
as \(r\to\infty\), where \(\Rightarrow\) denotes weak convergence in $D^m[0,\infty)$ and $\text{BM}(0,\Gamma)$ is $m$-dimensional Brownian motion with drift zero, covariance matrix \(\Gamma\), and zero initial state. 
\newline

\noindent \textbf{Additional standard assumptions.} The following additional assumptions on the primitive processes $E(\cdot),F^1(\cdot),\ldots,F^n(\cdot)$ are standard in queueing theory. First, the exogenous arrival process for each class $i=1,\ldots,m$ is a \textit{renewal} process: we denote by $\lambda_i$  the average exogenous arrival rate into class $i$ (that is, the reciprocal of the mean inter-arrival time); also, for each class $i$ that has a non-null exogenous input process (that is, $\lambda_i>0$), we assume that inter-arrival times have a finite variance $\eta_i^2>0$, and we set $\eta_i=0$ for all other classes.

Second, service times for activity $j$ are assumed to be independent and identically distributed (i.i.d.), with mean $\mu_j^{-1}>0$ and variance $\sigma_j^2>0$ ($j=1,\ldots,n$). Finally, recalling that $b(j)$ is the buffer served by activity $j$, we assume the following version of \textit{Markovian job routing}: after a job of class $b(j)$ completes a type $j$ service, the probability that it will next become a class $i$ job is $P_{ij}$, independent of all previous history, where $P$ is a nonnegative $m \times n$ matrix with column sums $ \leq 1$.

These additional assumptions justify or imply the functional central limit theorems \eqref{eq:FCLT}, and they also allow one to compute the asymptotic mean vectors and covariance matrices $\lambda,\Gamma^0,R^1,\Gamma^1,\cdots,R^n,\Gamma^n$ in terms of more basic system parameters. Those calculations will be reviewed in Appendix \ref{Additional}. The resulting formulas provide a satisfying element of concreteness, but the added structure is \textit{not} essential for the analytical method developed here, which depends only on the independence of the primitive processes $E(\cdot),F^1(\cdot),\ldots,F^n(\cdot)$, the gross structure of those processes that was assumed earlier, and validity of the functional central limit theorems \eqref{eq:FCLT}.
\newline

\noindent \textbf{Admissible policies and the system manager's objective.} Taking the sample paths of \(E(\cdot)\) and \(F^j(\cdot)\) to be right-continuous by convention, we require that a control policy \(T(\cdot)\) be non-decreasing with $T(0)=0$, and that it be non-anticipating with respect to the queue length process \(Q(\cdot)\) defined via \eqref{eq:queueing-dynamics}. Another restriction on a policy \(T(\cdot)\) is the following: if $b(j)=i$ (that is, activity $j$ serves buffer $i$), then $T_j(\cdot)$ increases only when $Q_i(\cdot)>0$. This ensures that each component of \(Q(\cdot)\) remains nonnegative at all times. Finally, readers are reminded of the capacity constraints \eqref{eq:capacity-constraint} that an admissible policy must satisfy, which imply that each component of \(T(\cdot)\) is absolutely continuous with uniformly bounded slope.

Turning now to system economics, we take as given an $m$-vector of \textit{holding cost rates} $h \geq 0$ and a $p$-vector of \textit{idleness cost rates} $c \geq 0$. The system manager chooses a control policy to
\begin{equation}
 \text{minimize  } \mathbb{E} \left \{ \int_{0}^{\infty} e^{- \rho t} [ \, h \cdot Q(t) \, dt + c\cdot dI(t) \, ] \right \},
\label{eq:objective}
\end{equation}
where $\rho>0$ is an interest rate for discounting, also given as data, and $\cdot$ denotes inner product as usual. As we shall see later, the ``idleness cost rates" provide a way of representing penalties that may be incurred for denying access to arriving jobs of various classes.

\subsection{Open multiclass queueing networks}
\label{sec:MCQN} An important special case of our general model is the open queueing network described in Section 2 of \citet{harrison1988brownian}. There are $p$ single-server stations that process jobs of $m$ different classes, and there is a one-to-one correspondence between job classes and activities. That is, $n=m$ and activity $j \in \{1,\ldots,m\}$ consists of server $s(j)$ providing services for class $j$ jobs. It is usual in the literature of queueing theory to impose the additional, ``standard" assumptions spelled out immediately above, and we shall do so here.

Readers should note that, for a conventional queueing network model with $n=m$, the matrix $P$ that was introduced above is the \textit{transpose} of the usual Markov transition matrix: in our notational system, the probability that a class $j$ job next becomes a class $i$ job after completing service is $P_{ij}$, and the probability that a class $j$ job leaves the system after completing service is $1-\sum_{i=1}^{m}P_{ij}$.

When we speak of an \textit{open} network model, that means the following: $\lambda_i>0$ for at least one class $i$, and $P$ has spectral radius $<1$ (thus its transpose is a \textit{transient} Markov matrix), which ensures that all jobs eventually leave the system. 

A simple example of an open queueing network is the criss-cross example pictured in Figure \ref{fig:criss-cross}, which was introduced by \cite{harrison1989scheduling} and has been further analyzed in a number of later papers, including \cite{martins1996heavy} and \citet{ata2024singular}. Here one has $p=2$ single-server stations processing jobs of $m=3$ classes. Jobs arrive to classes $1$ and $2$ according to independent Poisson processes with rates $\lambda_1$ and $\lambda_2$, respectively, and both of those classes are served at station $1$. After completing service, class $1$ jobs leave the system, whereas class $2$ jobs make a transition to class $3$ and are then served at station $2$. Service times for each class $i$ are exponentially distributed with mean service rate  $\mu_i>0$, and for maximum simplicity, we assume that any service can be interrupted at any time and resumed later without any efficiency loss. 

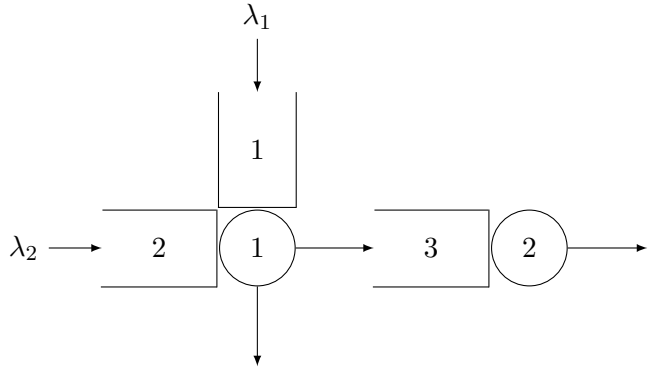
\begin{figure}[ht]
\centering
\begin{tikzpicture}[start chain=going right,>=latex,node distance=1pt]
		\node[three sided,minimum width=1.5cm,minimum height = 1cm,on chain] (wa) {2};

		\node[draw,circle,on chain,minimum size=1cm] (se) {1};
		\node[three sided2,minimum width=1cm,minimum height = 1.5cm,on chain]  (wa4)[above =of 
		se]  {1};
		
		\node[three sided,minimum width=1.5cm,minimum height = 1cm,on chain]  (wa5)[right =of 
		se, xshift=1cm]  {3};
		\node[draw,circle,on chain,minimum size=1cm] (se2) {2};

		\draw[->] (se) edge node[above] {} (wa5.west);
		\draw[->] (se2.east) -- +(30pt,0);
		\draw[->] (se.south) -- +(0,-30pt);
		
		\draw[<-] (wa.west) -- +(-20pt,0) node[left] {$\lambda_2$};
		\draw[<-] (wa4.north) -- +(0,20pt) node[above] {$\lambda_1$};
		
	\end{tikzpicture}
\caption{The criss-cross network.}
\label{fig:criss-cross}
\end{figure}

Inputs to classes $1$ and $2$ are uncontrollable, and there is a positive holding cost rate for jobs of each class. A system manager decides, at each point in time, which job class to serve at each station, including the possibility of idleness for either server. Because we assume a positive holding cost rate for each class, and because only class $3$ jobs are served at station $2$, an optimal policy will obviously keep server $2$ busy serving class $3$ whenever buffer $3$ is non-empty. Thus the problem boils down to deciding, at each point in time, whether to serve a class $1$ job at station $1$, serve a class $2$ job there, or idle server $1$. 

The arrival rates and mean service rates that we assume for the criss-cross model are the same ones used in the previous studies referred to above, namely,
\begin{equation}
(\lambda_1,\lambda_2, \lambda_3)=( 1,0.95, 0) \quad \text{and} \quad (\mu_1,\mu_2,\mu_3)=(2,2,1).
\label{eq:criss-crossData}
\end{equation}
Analysis of this example will be resumed later in Sections \ref{sec:CrissCross} and \ref{sec:criss-cross}, where the main development assumes the holding cost structure $(h_1,h_2, h_3)=( 1.5,1,1)$; \cite{harrison1989scheduling} assumed $(h_1,h_2, h_3)=( 1,1,1)$ in their original treatment of the criss-cross example, and that simpler case will be briefly considered as well. Finally, we take the idleness penalty rates to be $(c_1,c_2)=( 0,0)$. 

\subsection{Alternative processing modes}
\label{sec:AltModes}
Consider now a model in which all activities are processing activities, as opposed to input activities, but now $n>m$. That is, the number of activities available to the system manager is strictly larger than the number of buffers or job classes. In such a model there necessarily exists at least one class that can be served via two or more activities, which gives rise to the term ``alternative processing modes." In this setting, the term ``class $i$ service" would be ambiguous, so we say ``activity $j$ service" instead.

\begin{figure}[ht]
   \centering
\begin{tikzpicture}[node distance=1cm and 2cm]
    \node[three sided2,minimum width=1cm,minimum height = 1.5cm] (1) {1};
    \node[three sided2,minimum width=1cm,minimum height = 1.5cm, right=of 1] (2) {2};
    \node[three sided2,minimum width=1cm,minimum height = 1.5cm, right=of 2] (3) {3};

    \node[draw,minimum size=1cm,circle, below=of 1] (c1) {1};
    \node[draw,circle,minimum size=1cm, below=of 2] (c2) {2};
    \node[draw,circle,minimum size=1cm, below=of 3] (c3) {3};

    \draw[->] (1) edge node[left] {$\mu_1$} (c1);
    \draw[->] (2) edge node[below right] {$\mu_3$} (c2);
    \draw[->] (3) edge node[left] {$\mu_4$} (c3);
    \draw[->] (1) edge node[left] {$\mu_2$\phantom{1}} (c2);
    \draw[->] (3) edge node[near start,above] {$\mu_5$} (c1);
    \draw[->] (c1.south) -- +(0,-30pt);
    \draw[->] (c2.south) -- +(0,-30pt);
    \draw[->] (c3.south) -- +(0,-30pt);
\draw[<-] (1.north) -- +(0,20pt) node[above] {$\lambda_1$};
  \draw[<-] (2.north) -- +(0,20pt) node[above] {$\lambda_2$};
  \draw[<-] (3.north) -- +(0,20pt) node[above] {$\lambda_3$};
\end{tikzpicture}
\caption{Parallel-server example of \cite{pesic2016dynamic}}
\label{fig:pesic}
\end{figure}
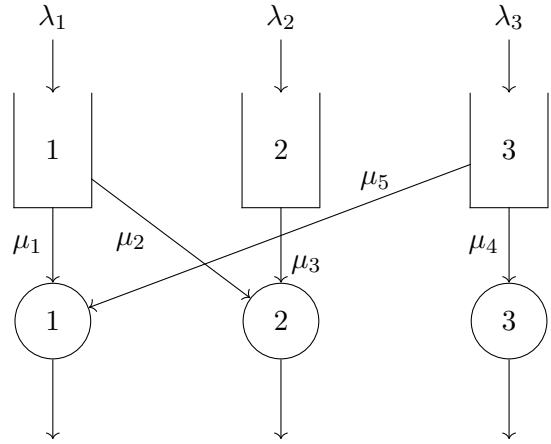

Figure \ref{fig:pesic} pictures a system with three servers and three job classes, each class having its own exogenous input process, and the following additional special structure: jobs of each class require just one service before they exit the system, which is expressed mathematically by the property $P_{ij}=0$ for all $i$ and $j$. Models having this structure are called \textit{parallel-server systems} in the literature. In this particular example, introduced by \cite{pesic2016dynamic}, two of the three job classes can be processed by either of two servers, as shown in the figure, so there are a total of five processing activities. As in the previous criss-cross example (Figure \ref{fig:criss-cross}), we assume Poisson input processes and exponential service time distributions, but this is done only to simplify exposition; the method developed here is distribution free, depending only on first-moment and second-moment data of the original queueing model. The arrival rates and mean service rates that we set for the Pesic-Williams example are as follows:
\begin{eqnarray}
    (\lambda_1,\lambda_2, \lambda_3) = (1.95, 0.95, 0.95) \ \text{ and } \  \mu =(1,2,2,1,1). 
\label{eqn-pesic-williams-arrival-service-rates}
\end{eqnarray} 
Analysis of this example will be continued in Sections \ref{sec:PesicWilliams} and \ref{sec:parallel-server}, where two different holding cost structures are considered. For one of them the EWF admits a pathwise optimal solution, but for the other one it does not. We focus primarily on the case without a pathwise optimal solution, reviewing just briefly the results of \citet{pesic2016dynamic} for the simpler case. Finally, as in our criss-cross example, inputs are not controllable in the Pesic-Williams example, and the idleness penalty rates are all zero ($c_1=c_2=c_3=0).$

\subsection{Dynamic routing and input control}
\label{sec:InputControl}

A subject of paramount importance in real-world applications of queueing theory is dynamic control of inputs to a system, usually called ``congestion control" in digital communications. To show how that capability can be accommodated within the framework of our general model (see Section \ref{sec:general}), we consider the three-station example pictured in Figure \ref{fig:three-station}, which was introduced by \citet{harrison1996bigstep} and further studied by \citet{harrison1997dynamic}.

\begin{figure}[!ht]
\centering
\includegraphics[width=6in]{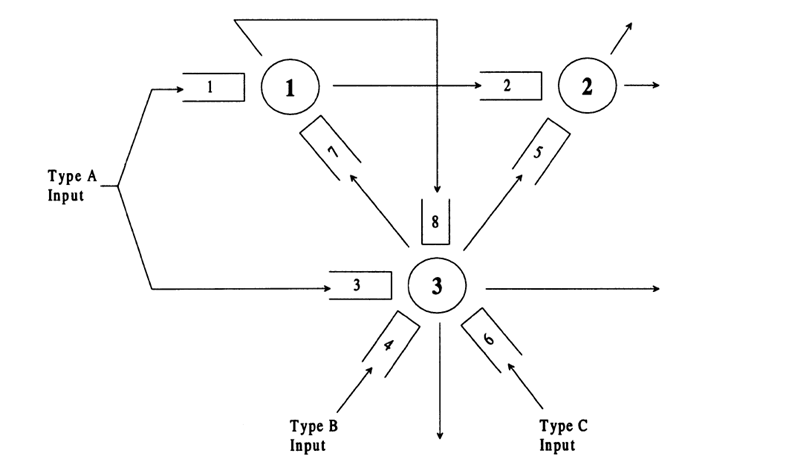}
\caption{Three-station example (copied from 
\citet{harrison1996bigstep}).}
\label{fig:three-station}
\end{figure}

In this example there are jobs of three different types A, B and C that arrive from the outside world according to independent Poisson processes; their average arrival rates are 0.5, 0.25 and 0.25, respectively. However, these are \textit{not} what we have called ``exogenous" arrival processes, because arrivals of each type can be rejected at the discretion of the system controller; see below for more about that.

Jobs of types B and C, if they are not rejected upon arrival, follow the predetermined routes shown in Figure  \ref{fig:three-station}, but for those of type A there are two processing modes available, and the routing decision must be made irrevocably at the moment of arrival. If a type A job is directed to the lower route, it requires just a single service at station 3, but if it is directed to the upper route, then two services are required at stations 1 and 2 in that order.

As shown in Figure  \ref{fig:three-station}, we define a different job class for each combination of job type and stage of completion, making eight classes in total. Each station serves multiple classes, so the system manager must make dynamic sequencing decisions (that is, order-of-service decisions), but also makes dynamic routing decisions for type A jobs. Finally, as stated above, the system manager can reject new jobs of any type upon arrival, incurring a type-dependent penalty for each rejected arrival. Those penalties must be weighed against class-dependent holding costs that provide a positive motivation for exercising input control when queues are large. As in our previous criss-cross and parallel-server examples, we assume independent Poisson arrival processes for the three job types in our three-station example, and exponential service time distributions for its eight job classes; as noted previously, those specific distributional assumptions simplify exposition but are inessential.

If dynamic routing and dynamic input control are eliminated from consideration, then there are eight processing activities in our three-station example, corresponding to service of the eight job classes identified in Figure  \ref{fig:three-station}, each conducted by one of three servers. To extend that formulation, it is conceptually easiest to speak in terms of three fictional ``input servers" numbered 4, 5 and 6, who generate arrivals of types A, B, and C, respectively. We associate with the input servers four 
input activities, as follows:
\newline

\indent \indent \indent  activity 9  = creation of class 1 jobs (by server 4),

\indent \indent \indent activity 10 = creation of class 3 jobs (by server 4),

\indent \indent \indent activity 11 = creation of class 4 jobs (by server 5),

\indent \indent \indent activity 12 = creation of class 6 jobs (by server 6).
\newline

Hereafter, when we say that server 4 is engaged in activity 9, this is understood to mean that the type A input process is turned on and any resulting arrival will be routed to buffer 1, and similarly for activities 10, 11 and 12. In the obvious way, idleness of an input server is interpreted to mean that the corresponding input process is turned off, or equivalently, that new arrivals of that type are being rejected.

For concreteness, we allow the input servers to work at less than full capacity, and allow server 4 to divide its time between activities 9 and 10, if doing so is deemed desirable. Such actions can be interpreted as randomized acceptance of new arrivals and randomized routing of type A arrivals, respectively. Similarly, servers 1, 2 and 3 are allowed to divide their time among activities available to them, processing several job classes simultaneously, provided that the total rate of effort allocation does not exceed the server's capacity of 1.

To repeat, all inputs are controllable in this model formulation. That is, external inputs to buffers 1, 3, 4 and 6 are represented as outputs of the four ``input activities" defined above. Thus, the exogenous input process $E$ in our general formulation is null for this particular model, so the associated vector of exogenous arrival rates is $\lambda=0$. Analysis of this example will be continued in Sections \ref{sec:ThreeStation} and \ref{sec:three-station}, where we assume specific values for mean service rates, holding cost rates and rejection penalties.

\section{Centering and scaling in the heavy traffic regime}
\label{sec:Centering}

Returning now to the general model laid out in Section \ref{sec:general}, we take as given an $n$-vector of \textit{nominal activity rates} $\beta$ such that
\begin{equation}
\beta \geq 0 \quad \text{and} \quad A \beta = e,
\label{eq:nominal_rates}
\end{equation}
and impose the following \textit{heavy traffic assumption}: there exists a large scaling parameter \(r > 0\) such that
\begin{equation}
 \text{all components of} \quad \zeta := \sqrt{r} (\lambda - R\beta) \quad \text{are of moderate absolute value.}
\label{eq:heavy_traffic}
\end{equation}
The equality in \eqref{eq:nominal_rates} says that the nominal activity rates precisely consume the full capacity of each server. In the special case of \(\zeta = 0\), we say that the system is \textit{perfectly balanced} under the nominal control policy \(T^*(t) = \beta t, t \geq 0\), and because our scaling parameter \(r\) is large by hypothesis, one may interpret \eqref{eq:heavy_traffic} as saying that the system under study is \textit{approximately balanced} under nominal activity rates. That is, \eqref{eq:heavy_traffic} says that under the nominal control policy 
$T^*(\cdot)$, the total rate of flow into each buffer is approximately equal to the total outflow rate.
\newline

\noindent \textbf{Deviation controls.} Our analytical method is based on an implicit assumption that the system manager knows in advance, via crude first-order analyses of the kind illustrated in  Appendix \ref{Nominal}, that average activity rates close to \(\beta_1, \dots, \beta_n\) are nearly optimal over long time spans. To put that in different words, one may say that the vector $\beta$ constitutes a \textit{rough-cut operating plan}, from which only relatively minor deviations are expected. Thus we are led to re-express the system manager's policy as an \(n\)-vector of cumulative deviation controls \(Y(\cdot)\) as follows:
\begin{equation}
Y(t) = \beta t - T(t), \quad t \geq 0.
\label{eq:deviation_control}
\end{equation}
To repeat, the nominal activity rates \(\beta_1, \dots, \beta_n\) are treated here as given data, which is equivalent to saying that the preliminary analysis required for their determination (see Appendix \ref{Nominal} for examples) lies outside the scope of our analysis.
\newline

\noindent \textbf{Re-expressing the main system equation.} Using  \eqref{eq:centered_stochastic} and \eqref{eq:nominal_rates}, one sees that the basic relationship \eqref{eq:queueing-dynamics} can be rewritten in terms of centered processes and deviation controls as follows:
\begin{equation}
Q(t) = Q(0) + \hat{X}(t) + (\lambda - R \beta) t + R Y(t)  \geq 0, \quad t \geq 0,
\label{eq:queue_length}
\end{equation}
where
\begin{equation}
\hat{X}(t) = \hat{E}(t) - \sum_{j=1}^n \hat{F}^j(T_j(t)), \quad t \geq 0.
\label{eq:centered_queue}
\end{equation}
Finally, let us define the scaled processes
\begin{equation}
\tilde{Q}(t) = \frac{1}{\sqrt{r}} Q(rt), \quad \tilde{X}(t) = \frac{1}{\sqrt{r}} \hat{X}(rt), \quad \tilde{Y}(t) = \frac{1}{\sqrt{r}} Y(rt), \quad t \geq 0.
\label{eq:rescaled_processes}
\end{equation}
Combining \eqref{eq:queue_length} and \eqref{eq:rescaled_processes} with the definition \eqref{eq:heavy_traffic} of \(\zeta\) gives us the following centered-and-scaled version of the basic system relationship \eqref{eq:queueing-dynamics}:
\begin{equation}
\tilde{Q}(t) = \tilde{Q}(0) + \tilde{X}(t) + \zeta t + R \tilde{Y}(t)  \geq 0, \quad t \geq 0.
\label{eq:centered_scaled_system}
\end{equation}

\noindent \textbf{Re-expressing capacity constraints.} To re-express the capacity constraints \eqref{eq:capacity-constraint} in terms of centered-and-scaled processes, we first observe the following: from the definition \eqref{eq:defineIdleness} of the cumulative idleness process $I(\cdot)$, plus \eqref{eq:nominal_rates}, \eqref{eq:deviation_control} and \eqref{eq:rescaled_processes}, one has that
\begin{equation}
\tilde{I}(t) = A \tilde{Y}(t), \text{  where  } \tilde{I}(t) = \frac{1}{\sqrt{r}} I(rt),\quad t \geq 0.
\label{eq:rescaled_capacity}
\end{equation}
Thus the capacity constraints \eqref{eq:capacity-constraint} are equivalently expressed as follows:
\begin{equation}
\tilde{I}(\cdot) \text{ is non-decreasing with } \tilde{I}(0) = 0.
\label{eq:capacity_again}
\end{equation}

\noindent \textbf{Accounting for nonbasic activities.} An activity $j \in \{1,\ldots,n\}$ will be called \textit{basic} if ${\beta_j>0}$ and \textit{nonbasic} if ${\beta_j=0}$. Nonbasic activities have no role in the system manager's rough-cut operating plan, but they may be used sparingly in exceptional circumstances. As will be seen in Appendix \ref{Nominal}, only one of the three illustrative networks discussed earlier in Section \ref{sec:general} involves nonbasic activities, that being the parallel-server example pictured in Figure \ref{fig:pesic}. There is just one nonbasic activity in that example (see Appendix \ref{Nominal} for specifics), that being a combination of server and job class that is ``inefficient" if one considers only mean service rates but may be useful when the corresponding server would otherwise be forced to idle.

For a nonbasic activity $j$, the definition \eqref{eq:deviation_control} simply reduces to $Y_j(\cdot)=-T_j(\cdot)$. Thus, when a network model involves nonbasic activities, the monotonicity constraints \eqref{eq:capacity_again} must be augmented to include the requirement that $-Y_j(\cdot)$ be non-decreasing for nonbasic $j$. Using the approach adopted in \citet{harrison2003broader}, this can be accomplished as follows. First, let the basic activities be numbered $1,\ldots,b$, where $1 \leq b \leq n$, and then partition the $p \times n$ capacity consumption matrix $A$ as $A=(B,N)$, where $B$ is $p \times b$ and $N$ is $p \times (n-b)$. Next, define 
\begin{equation}
   K=
  \left[ {\begin{array}{cc}
   B & N \\
   0 & -I \\
  \end{array} } \right] \text{  where $I$ is the $(n-b) \times (n-b)$ identity matrix},
\label{eq:define_K}
\end{equation}
and 
\begin{equation}
\tilde{U}(t) = K \tilde{Y}(t), \quad t \geq 0.
\label{eq:define_U-tilde}
\end{equation}
Then the augmented version of \eqref{eq:capacity_again} that we require is the following:
\begin{equation}
\tilde{U}(\cdot) \text{ is non-decreasing with } \tilde{U}(0) = 0.
\label{eq:Monotonicity}
\end{equation}

\noindent \textbf{Re-expressing the objective.} As a final element of the ``heavy traffic" parameter regime under study here, we assume that the interest rate $\rho$ appearing in \eqref{eq:objective} is small, which means that long time spans are relevant in choosing a control. More precisely, we assume that
\begin{equation}
\gamma := r \rho > 0 \quad \text{is of moderate size,}
\label{eq:smallrho}
\end{equation}
where $r$ is the large scaling parameter introduced in Section \ref{sec:Centering}. Using this definition of $\gamma$, plus the definition of $\tilde{Q}(\cdot)$ in \eqref{eq:rescaled_processes}, the definition of $\tilde{Y}(\cdot)$ via \eqref{eq:deviation_control} and \eqref{eq:rescaled_processes}, and the definition of $\tilde{U}(\cdot)$ via \eqref{eq:rescaled_capacity} and \eqref{eq:define_U-tilde}, we can make the change of variable $s=t/r$ in \eqref{eq:objective} to re-express our system manager's objective as follows: 
\begin{equation}
 \text{minimize  } \mathbb{E} \left \{ \int_{0}^{\infty} e^{- \gamma s} [ \, \tilde{h} \cdot \tilde{Q}(s) \,  ds + \tilde{c} \cdot d\tilde{U}(s) \, ]  \right \},
\label{eq:scaledobjective}
\end{equation}
where
\begin{equation}
 \tilde{h} = (r^{3/2}h_1,\ldots,r^{3/2}h_m) \text{ \quad and \quad } \tilde{c} = (r^{1/2}c_1,\ldots,r^{1/2}c_p, 0,\ldots,0).
\label{eq:scaledcoeffs}
\end{equation}
In the definition \eqref{eq:scaledcoeffs} of $\tilde h$ and $\tilde{c}$, we have rescaled the $m$-vector $h$ and the $p$-vector $c$ by factors of $r^{3/2}$ and $r^{1/2}$, respectively, but also appended to $c$ a final set of $n-b$ zeros, reflecting the fact that there are no penalties associated with the introduction of nonbasic activities.

The re-scaling of space and time that we have used in this section can be interpreted as a change in units of measurement: we have changed by a factor of $\sqrt r$ the units in which buffer contents and flow quantities are measured, and simultaneously changed by a factor of $r$ the units in which time is measured. The different powers of $r$ that appear in the expressions for $\tilde h$ and $\tilde c$ reflect the fact that $h$ is expressed in units like dollars per job per minute, whereas $c$ is expressed in units like dollars per minute of idle time. 
\newline

\noindent \textbf{Summary.} To summarize, the system manager chooses a scaled deviation control \(\tilde{Y}(\cdot)\) to  minimize the objective \eqref{eq:scaledobjective}, where $\tilde{Q}(\cdot)$ and $\tilde{U}(\cdot)$ are defined by \eqref{eq:centered_scaled_system} and \eqref{eq:define_U-tilde}, respectively. Of course, \(\tilde{Y}(\cdot)\) must be suitably adapted, and it must satisfy the monotonicity constraints embodied in \eqref{eq:Monotonicity}.

\section{Brownian analog of the general control problem}
\label{sec:BrownianAnalog}
In this section we develop a natural diffusion approximation, or Brownian approximation, or Brownian analog, for the general control problem that was originally described in Section \ref{sec:general}, and was then re-expressed in scaled terms via \eqref{eq:centered_scaled_system}, \eqref{eq:Monotonicity} and \eqref{eq:scaledobjective}. In this reformulation we denote by $Z(\cdot), X(\cdot), Y(\cdot)$ and $U(\cdot)$ the analogs of processes $\tilde{Q}(\cdot), \tilde{X}(\cdot), \tilde{Y}(\cdot)$ and $\tilde{U}(\cdot)$ that either have been defined already, or else will be defined shortly, for the queueing network model of original interest. This notational system involves a re-use of the letter $Y$, originally defined via \eqref{eq:deviation_control}, to denote a related but different process. Our hope and belief is that this re-use will cause no confusion for readers who are alert to the practice.

\subsection{The Brownian motion $X(\cdot)$ that approximates $\tilde{X}(\cdot)$}

The core of our alternative formulation is a Brownian approximation for the centered-and-scaled process $\tilde{X}(\cdot)$ defined via \eqref{eq:rescaled_processes}. This approximation is justified by essentially the same heuristic argument that was made in Section 4 of \citet{harrison1988brownian}. To begin, let us define scaled versions of the centered primitive processes defined earlier via \eqref{eq:centered_stochastic}:
\begin{equation}
\tilde{E}(t) = \frac{1}{\sqrt{r}}\hat{E}(rt) \quad \text{and} \quad \tilde{F}^j(t) = \frac{1}{\sqrt{r}}\hat{F}^j(rt) \quad \text{for  } t \geq 0 \text{   and   } j=1,\ldots,n.
\label{eq:rescaled_primitives}
\end{equation}
From \eqref{eq:centered_queue}, \eqref{eq:rescaled_processes} and \eqref{eq:rescaled_primitives} one has that
\begin{equation}
\tilde{X}(t) = \tilde{E}(t) - \sum_{j=1}^n \tilde{F}^j(\tau_j(t)), \quad t \geq 0,
\label{eq:ScaledX}
\end{equation}
where
\begin{equation}
\tau_j(t) = \frac{1}{r}T_j(rt) \quad \text{for  } t \geq 0 \text{   and   } j=1,\ldots,n.
\label{eq:DefineTau}
\end{equation}
As noted previously in Section \ref{sec:Centering}, we interpret $\beta$ as a vector of average activity rates, or average capacity allocations, that are nearly optimal over long time spans, so it is natural to approximate $\tau_j(t)$ by $\beta_jt$. Combining that with \eqref{eq:ScaledX} and the functional central limit theorems \eqref{eq:FCLT}, one is led to approximate $\tilde{X}(t)$ by $B^0(t)+B^1(\beta_1t)+\ldots+B^n(\beta_nt)$, where $B^0(\cdot),B^1(\cdot),\ldots,B^n(\cdot)$ are mutually independent $m$-dimensional Brownian motions, each with zero drift and with covariance matrices $\Gamma^0, \Gamma^1, \ldots , \Gamma^n$, respectively. That is, we are led to the approximation $\tilde{X}(\cdot)\approx X(\cdot)$, where
\begin{equation}
X(\cdot)\text{   is  } \text{BM}(0,\Gamma) \quad \text{with} \quad \Gamma =\Gamma^0+\beta_1 \Gamma^1+ \ldots + \beta_n \Gamma^n.
\label{eq:DefineX}
\end{equation}

\subsection{The initial Brownian control problem (BCP)}
\label{sec:raw}
Substituting $X(\cdot)$ for $\tilde{X}(\cdot)$ in the scaled dynamic control problem formulated in Section \ref{sec:Centering}, we have the following key relationships for our Brownian system model:
\begin{equation}
Z(t) = Z(0) + X(t) + \zeta t + R Y(t) \geq 0, \quad t \geq 0,
\label{eq:BCP1}
\end{equation}
\begin{equation}
U(t) = KY(t), \quad t \geq 0,
\label{eq:BCP2}
\end{equation}
\begin{equation}
U(\cdot) \text{ is non-decreasing with } U(0) = 0,
\label{eq:BCP3}
\end{equation}
where $Z(0) \in \mathbb{R}_+^m$ is given as problem data. Again it should be emphasized that the displayed relationship \eqref{eq:BCP1} combines two elements of our BCP specification: first, the process $Z(\cdot)$ is defined by the equality in \eqref{eq:BCP1}, and second, the control $Y(\cdot)$ must be chosen so that $Z(\cdot)$ remains nonnegative at all times. Also, the control $Y(\cdot)$ must be non-anticipating with respect to the underlying Brownian motion $X(\cdot)$, which is the only source of uncertainty in the BCP. Finally, as the objective in our alternative Brownian formulation, we have the following obvious analog of \eqref{eq:scaledobjective}:
\begin{equation}
 \text{minimize  } \mathbb{E} \left \{ \int_{0}^{\infty} e^{- \gamma t} [ \, \tilde h \cdot Z(t) \, dt + \tilde c \cdot dU(t) \, ] \right \}.
\label{eq:Brownianobjective}
\end{equation}
Hereafter, this problem will be referred to as our \textit{initial BCP}, or BCP in initial form, to distinguish it from the \textit{modified BCP} introduced in Section \ref{sec:drift1} below.

\subsection{Equivalent workload formulation (EWF): The BCP in singular control form}
\label{sec:EWF}
A notable, non-standard feature of the initial BCP formulation immediately above is that it allows individual components of the control process $Y$ to have unbounded variation, and also to have no associated cost of control. In such cases, the system manager is able to enforce displacements in the state vector $Z$ that are instantaneous, costless and reversible, but these ``reversible displacements" can only be effected in certain directions. Assuming that reversible displacements are achievable in at least one direction, the BCP outlined in Subsection \ref{sec:raw} can be reduced to another, equivalent stochastic control problem of lower dimension. That lower dimensional problem, called an \textit{equivalent workload formulation} (EWF) in the literature, is one of \textit{singular} stochastic control, cf. \cite{karatzas1983}, and the reduction of the system manager's problem from its original form to its EWF is described by the phrase \textit{state space collapse}. The general theory of that problem reduction, developed by \citet{harrison1997dynamic}, need not be recapitulated here, but the following three examples will help to put in perspective the alternative formulation to be developed in Section \ref{sec:drift1}.

\section{Brownian control problems for three examples}
\label{sec:ThreeExamples}

In this section we return to the three examples of network control problems that were introduced in Section \ref{sec:general}. For each of them in succession, we briefly recapitulate results that have been derived in earlier work, specifically via the following stages of analysis: (a) determine an appropriate choice of the nominal activity rates $\beta_1, \ldots, \beta_n$ that were introduced in Section \ref{sec:Centering}; (b) formulate a corresponding $m$-dimensional Brownian control problem, referred to as the \textit{initial BCP} in Subsection \ref{sec:raw}; and then (c) derive the lower-dimensional EWF described in Subsection \ref{sec:EWF}. Analysis of these examples will be continued later in Section \ref{sec:Tests}.

\subsection{Criss-cross example} 
\label{sec:CrissCross}
To formulate a BCP for the criss-cross network pictured in Figure \ref{fig:criss-cross}, with the arrival rates and mean service rates specified via  \eqref{eq:criss-crossData}, we take our scaling parameter to be $r=400$ and the interest rate for discounting in the original (unscaled) model to be $\rho = 0.01$. Thus the interest rate for the corresponding BCP is $\gamma = r \rho= 4$.

The initial BCP for this network (see Subsection \ref{sec:raw}), has a three-dimensional state vector $Z(t)$, which corresponds to the scaled queue length process $\tilde{Q}(t)$ introduced in Section \ref{sec:Centering}. The data associated with the BCP include the matrices
\begin{equation}
   R =
  \left[ {\begin{array} {rrr}
   2 & 0 & 0 \\
   0 & 2 & 0 \\ 
   0 & -2 & 1
  \end{array} } \right] 
, \quad
   K = A =
  \left[ {\begin{array} {rrr}
   1 & 1 & 0 \\
   0 & 0 & 1 \\
  \end{array} } \right]
  \quad \text{and} \quad
\Gamma =
  \left[ {\begin{array} {rrr}
   2 & 0 & 0 \\
   0 & 2 & -1 \\ 
   0 & -1 & 2
  \end{array} } \right]. 
\label{eq:CrissCrossMatrices}
\end{equation}
Nominal allocation rates $\beta = (0.5, 0.5, 1)$ are proposed for the criss-cross network in Appendix \ref{Nominal}, so the three-dimensional Brownian motion $X$ underlying the criss-cross BCP has drift vector $\zeta$ and covariance matrix $\Gamma$, where 
\begin{equation*}
    \zeta = \sqrt{r} (\lambda - R \beta) = (0,-1,0)
\end{equation*}
and $\Gamma$ is given by \eqref{eq:CrissCrossMatrices}. The holding cost rates are initially taken to be $h=(h_1,h_2,h_3) =(1.5,1,1)$, and by assumption there are no idleness penalties, that is, $c=(c_1,c_2)=(0,0)$. Finally, because holding costs are the only considerations in the system manager's objective, and because all holding cost rates are scaled by the same factor in our formulation of the approximating BCP, \textit{that scaling is irrelevant}, and one can simply proceed as if $\tilde{h}=h$.

The corresponding EWF, originally derived by \cite{harrison1989scheduling} and recapitulated in Section 6 of \citep{ata2024singular}, has the two-dimensional state vector 
\begin{equation}
W(t)=MZ(t), \ t\geq0, 
\end{equation}
where
\begin{equation}
   M=
  \left[ {\begin{array} {ccc}
   0.5 & 0.5 & 0 \\
   0 & 1 & 1 \\
  \end{array} } \right].
  \label{eq:defineM}
\end{equation}
The main system equation for the EWF specifies that $W(t)=\eta(t)+U(t)$  for $t\geq0$, where $\{\eta(t),\,t\geq0\}$ is a two-dimensional Brownian motion with drift vector 
\begin{eqnarray}
    \xi = M \zeta = (-0.5, -1)
\end{eqnarray}
and covariance matrix
\begin{equation*}
   \Sigma=M \, \Gamma M^{'}=
  \left[ {\begin{array} {cc}
   1 & 0.5 \\
   0.5 & 2 \\
  \end{array} } \right].
\end{equation*}
Here $\{U(t), \ t\geq0\}$ is a non-decreasing control process whose components represent (scaled versions of) the cumulative idleness experienced by the network's two servers. Except for scale factors, $W_i(t)$ represents the expected time required for server $i$ to complete the processing of all jobs present in the system at time $t$ ($i=1,2$), so in this example the term ``workload" is entirely appropriate.

In the equivalent workload formulation, the system manager must choose a non-decreasing control $U$ such that both components of $W$ remain nonnegative, and holding costs are continuously incurred at rate $\tilde h(W(t))$, where 
\begin{equation}
\tilde h(w)=\min\{\tilde h \cdot z: Mz=w, \ z \geq 0\}, \, w \geq 0.
\label{eq:cost}
\end{equation}
To repeat, the system manager controls the workload process $W$ through idleness decisions, and then \eqref{eq:cost} is interpreted as follows: at each time $t$, the queue length vector $Z(t)$ can be taken as any $z \in \mathbb{R}_+^3$ such that $Mz=W(t)$, and the system manager naturally chooses the $z$ value with the minimum associated total holding cost. Readers are reminded that the tildes occurring in \eqref{eq:cost} reflect the following fact: in general, the objective function \eqref{eq:Brownianobjective} for our approximating Brownian control problem, and hence also for its equivalent workload formulation, involves the vectors $\tilde h$ and $\tilde c$ of \textit{scaled} cost rates. For the example currently under discussion, however, the distinction between scaled and unscaled cost rates is actually irrelevant, as noted previously in this subsection.

As mentioned earlier, the criss-cross network was introduced by \cite{harrison1989scheduling}. Those authors assumed the scaled holding cost rates $\tilde h = (1,1,1)$, which corresponds to minimizing the total number of jobs in the system. \cite{martins1996heavy} revisited the same model but allowed more general holding cost parameters. The latter authors identified five mutually exclusive and collectively exhaustive cases with respect to the holding cost parameters, and labeled them Cases I, IIA, IIB, IIC and IID. The cost structure assumed by \cite{harrison1989scheduling} is an instance of Case IIA, and they proposed an effective policy for use in that case. \cite{martins1996heavy} proved the asymptotic optimality of that policy; see also \cite{budhiraja2005large}. \cite{martins1996heavy} further advanced conjectures about the remaining cases, two of which (for Cases IIB and IIC) were rigorously justified in later papers by \cite{budhiraja2008optimal} and \cite{budhiraja2017construction}. \cite{ata2024singular} applied the computational method developed in their earlier paper \cite{ata2023drift} to numerically solve problems in all four subcases of Case II; see Figure 5 of \cite{ata2024singular}. The optimal policies displayed in that figure are consistent with the conjectures of \cite{martins1996heavy}. 

The holding cost vector $\tilde{h} =(1.5,1,1)$ that we have assumed in this section is an instance of what \cite{martins1996heavy} identified as Case IIC, and with these specific cost data, the minimization problem \eqref{eq:cost} has the following optimal solution:
\begin{eqnarray}
     z_2 = 2w_1 \wedge w_2, \ z_1 = (2w_1 -w_2)^+ \ \text{and } z_3 = (w_2 - 2w_1)^+;
     \label{eqn-ssc-criss-cross}
\end{eqnarray}
see Equations (8.5a)-(8.5b) of \cite{martins1996heavy}. Combining \eqref{eq:cost} and \eqref{eqn-ssc-criss-cross}, we then have the following: 
\begin{equation*}
  \tilde h(w)=
  \left\{ {\begin{array} {lc}
   w_2, & \text{ if } \ w_2 \geq 2w_1, \\
   3w_1 - 0.5w_2, & \text{ if } \ w_2 < 2w_1; \\
  \end{array} } \right.
\end{equation*}
see also Equation (8.4) of \cite{martins1996heavy}. The solution of the EWF for this particular instance of the criss-cross example is portrayed graphically in Figure \ref{fig:crisscross-EWF} below.

\begin{figure}[ht]
    \centering
    \includegraphics[width=0.5\linewidth]{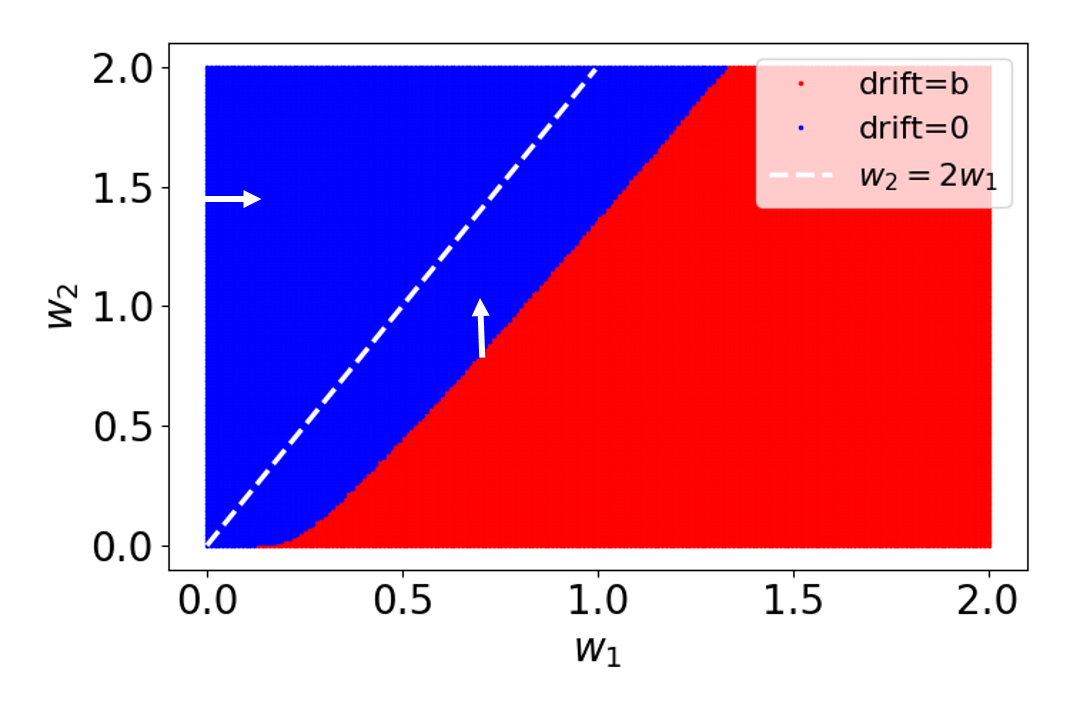}
    \caption{Optimal control policy for the criss-cross EWF (copied from \cite{ata2024singular}).}
    \label{fig:crisscross-EWF}
\end{figure}

\cite{martins1996heavy} interpreted the policy displayed in Figure \ref{fig:crisscross-EWF} as follows: First, idle server 1 only if $W$ reaches the vertical axis (that is, buffers 1 and 2 are both empty). Second, idle server 2 whenever $W$ falls in the red region; this is accomplished by having server 1 give priority to class 1 over class 2, which starves server 2. Third, give class 1 priority at station 1 if $W$ falls in the blue region and $Q_3 > \delta $; otherwise give priority to class 2. Here $ \delta >0 $ is a safety stock parameter. The best value of $ \delta $ can be determined via simulation experiments, and we expect it to be small relative to our spatial scaling parameter of $\sqrt{r}$ (see Section \ref{sec:Centering}). That is, we expect the best threshold value to be near zero when re-expressed in scaled terms.

\subsection{Pesic-Williams parallel-server example} 
\label{sec:PesicWilliams}
To formulate a BCP for the Pesic-Williams example pictured in Figure \ref{fig:pesic}, with arrival rates and mean service rates specified via \eqref{eqn-pesic-williams-arrival-service-rates}, we again take $r=400$ and $\rho=0.01$, so the interest rate for discounting in the BCP is $ \gamma = r \rho = 4$. The nominal activity rates proposed for the Pesic-Williams example in Appendix \ref{Nominal} are
\begin{eqnarray}
    \beta = (1, 0.5, 0.5, 1, 0),
\end{eqnarray}
and that proposal is consistent with Equation (9.2) of \cite{pesic2016dynamic}. Notably, activity 5 is nonbasic, that is, it is not used at all in the nominal processing plan. 

The initial BCP for the Pesic-Williams example has a three-dimensional state vector $Z(t)$, and its data include the following matrices:
\begin{equation}
    R =
  \left[ {\begin{array} {rrrrr}
   1 & 2 & 0 & 0 & 0 \\
   0 & 0 & 2 & 0& 0\\ 
   0 & 0 & 0 & 1 &1
  \end{array} } \right],  \ \ 
 A =
  \left[ {\begin{array} {rrrrr}
   1 & 0 & 0 & 0 & 1 \\
   0 & 1 & 1 & 0& 0\\ 
   0 & 0 & 0 & 1 &0
  \end{array} } \right]
\, \text{ and } \, 
    K = \left[ {\begin{array} {rr}
   A &   \\
    & -1 
  \end{array} } \right].
  \label{PWmatrices}
\end{equation}
The drift  vector for the underlying three-dimensional Brownian motion $X$ is $\zeta = (-1,-1,-1)$, and its covariance matrix is  
\begin{equation}
   \Gamma =
  \left[ {\begin{array} {ccc}
   4 & 0 & 0 \\
   0 & 2 & 0 \\
   0 & 0 & 2 \\
  \end{array} } \right].
\end{equation}
Finally, two different choices of the holding cost vector $h$ will be considered below, and there are no idleness penalties in the Pesic-Williams example ($c=0$). As in our earlier treatment of the criss-cross example (Subsection \ref{sec:CrissCross}), we shall ignore the distinction between scaled and unscaled holding costs, \textit{proceeding as if $\tilde{h}=h$}, because that scaling is irrelevant when holding costs are the system manager's only concern.
 
For the parallel-server example under discussion here, \cite{pesic2016dynamic} derive the corresponding EWF. It features the two-dimensional state vector $W(t) = M Z(t), \,t \geq 0$, where
\begin{equation}
   M=
  \left[ {\begin{array} {ccc}
   1 & 1 & 0 \\
   0 & 0 & 1 \\
  \end{array} } \right].
  \label{eq:defineM-pesic-williams}
\end{equation}
The main system equation for the EWF specifies that $W(t) = \eta (t) + G U(t)$ for $t \geq 0$, where $\{ \eta(t), \, t \geq 0 \}$ is a two-dimensional Brownian motion with drift vector $\xi = M \zeta = (-2, -1)$ and covariance matrix
\begin{equation*}
   \Sigma= M \Gamma M^{'} =
  \left[ {\begin{array} {cc}
   6 & 0 \\
   0 & 2 \\
  \end{array} } \right].
\end{equation*}
Here $U$ is a four-dimensional control with all components non-decreasing, and 
\begin{equation*}
   G=
  \left[ {\begin{array} {cccr}
   1 & 2 & 0 & 1 \\
   0 & 0 & 1 & -1 
  \end{array} } \right];
\end{equation*}
$U_1, U_2, U_3$ (and the first three columns of $G$) correspond to idling the three servers, whereas $U_4$ and the last column of $G$ correspond to introduction of activity 5 (the nonbasic activity).
  
\cite{pesic2016dynamic} further reduce the EWF to arrive at a simpler problem they refer to as the reduced equivalent workload formulation. For this example, that reduction amounts to the following: for purposes of optimal control, and for any vector $\tilde h \geq 0$ of scaled holding costs, one can set $U_1(\cdot)=U_2(\cdot)=0$ without loss of optimality. That is, an optimal policy for the EWF need never idle server 2, and insertion of the nonbasic activity 5 (server 1 processing jobs from buffer 3) is always preferable to idling server 1. 

As mentioned earlier for the criss-cross network, holding costs are incurred at rate $\tilde{h}(W(t))$ in the EWF formulation, where
$\tilde h(\cdot)$ is defined via \eqref{eq:cost}. For this example, \cite{pesic2016dynamic} show that
\begin{equation}
    \tilde{h}(w) = \min ( \tilde{h}_1 , \tilde{h}_2 ) w_1 + \tilde{h}_3 w_2, \ w \geq0. \label{eq:PWcost}
\end{equation}
Under the assumptions that (i) $\tilde{h}_1 > \tilde{h}_2$ and (ii) $\tilde{h}_2 > \tilde{h}_3$, \cite{pesic2016dynamic} show that an optimal policy for control of the workload process $W$ is the one portrayed in Figure \ref{fig:pesic-williams-EWF}, where the notation $G^j$ refers to the $j^{th}$ column of the matrix $G$ displayed above. The corresponding optimal state process $Z$ is as follows:
\begin{equation}
    Z_1(t) = 0, \ Z_2(t) = W_1(t) \ \text{and} \ Z_3(t) = W_2(t), \ \ t \geq 0.
\label{eq:PWZ}
\end{equation}

The control policy described by Figure \ref{fig:pesic-williams-EWF} and \eqref{eq:PWZ} is interpreted as follows. First, server 2 gives priority to buffer 1 except when the content of buffer 1 falls below a critical level $\delta >0$ (small in scaled terms); the goal is to keep the content of buffer 1 positive but small compared to the content of buffer 2. Second, server 1 devotes its full capacity to service of buffer 1 except when buffer 1 is empty, in which case server 1 devotes its full capacity to buffer 3 (the nonbasic activity). Finally, server 3 works full time on buffer 3 except when buffer 3 is empty, in which case it idles.


\begin{figure}[ht]
    \centering
    \includegraphics[width=0.3\linewidth]{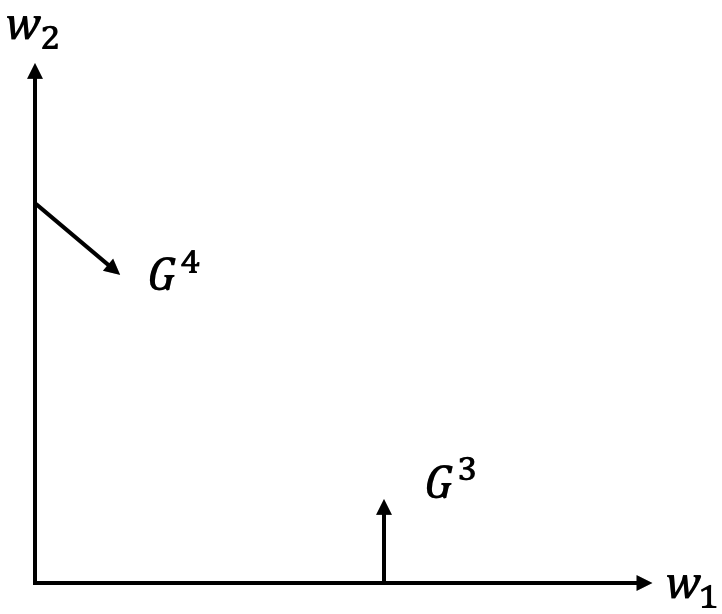}
    \caption{Optimal policy for control of $W$ in the Pesic and Williams example under assumptions (i) and (ii), adapted from Figure 5 of \cite{pesic2016dynamic}.}
    \label{fig:pesic-williams-EWF}
\end{figure}

Let us now consider the scaled holding cost vector $\tilde{h}  = (1,2,3) $, which violates assumptions (i) and (ii) above. For this case, \cite{pesic2016dynamic} state that ``We do not know how to solve the EWF in this case, as the solution may involve a free boundary on which the nonbasic activity is used." We have solved this problem instance numerically using the computational method proposed by \cite{kushner1991numerical}, and also solved it using the method developed in \cite{ata2024singular}. The very similar solutions obtained by those two methods are pictured in Figure \ref{fig:pesic-EWF}. The two methods also produce nearly identical objective values, differing by no more than $0.15 \%$.
\begin{figure}[htbp]
    \centering
    \begin{minipage}{0.48\textwidth}
        \centering
        \includegraphics[width=\textwidth]{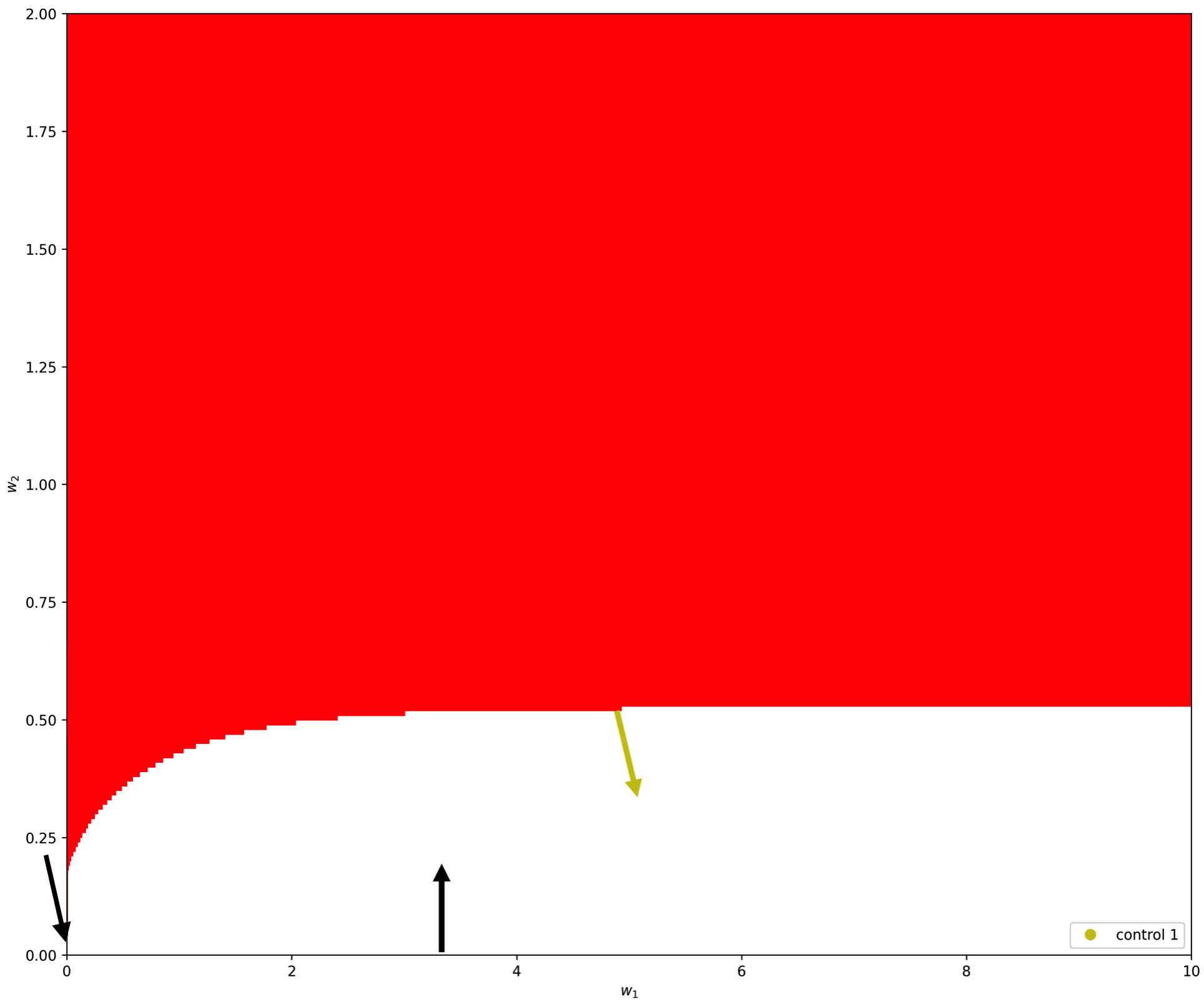}
        \par\vspace{4pt}\textbf{(a)} \citet{kushner1991numerical}.
    \end{minipage}\hfill
    \begin{minipage}{0.48\textwidth}
        \centering
        \includegraphics[width=\textwidth]{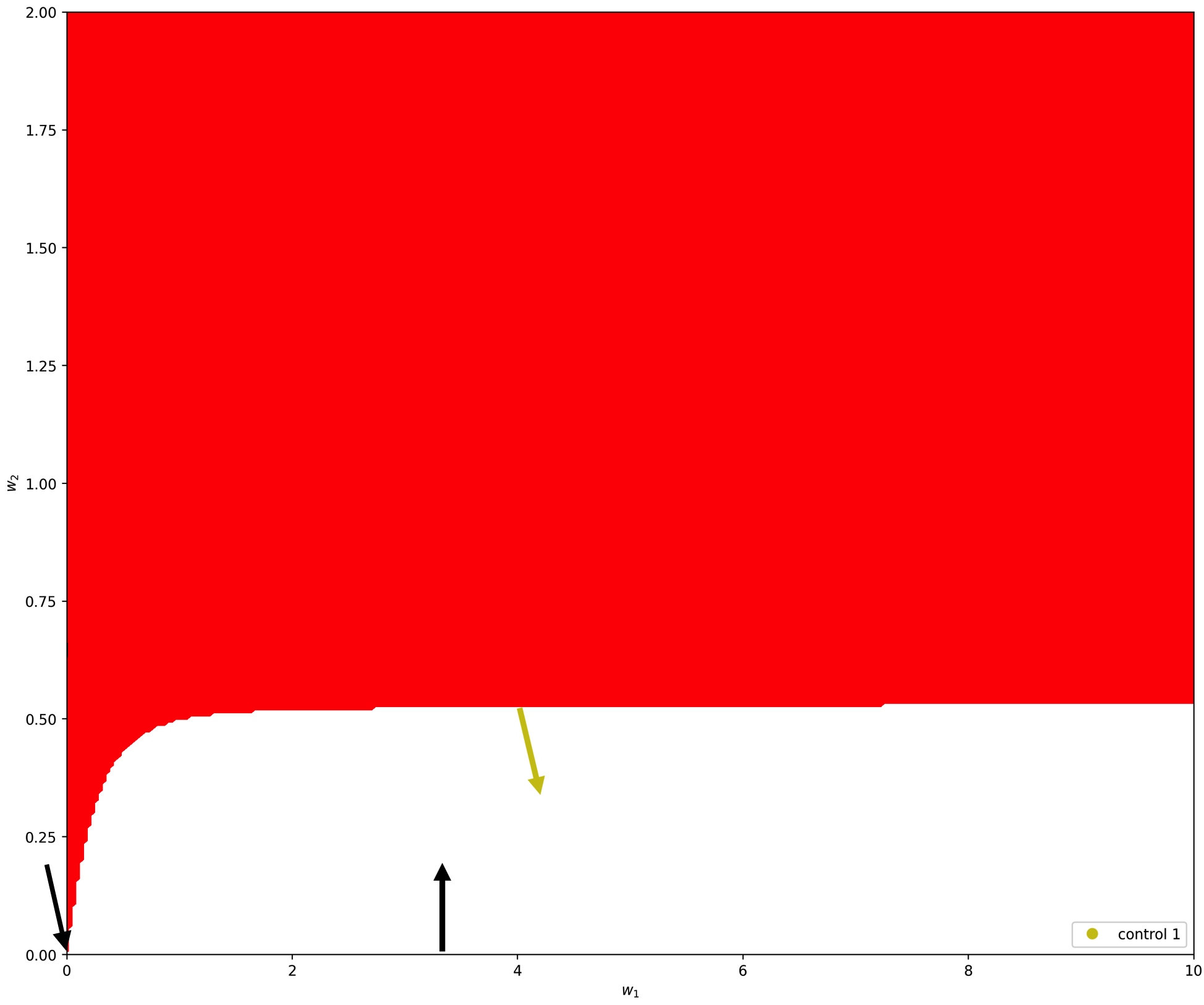}
        \par\vspace{4pt}\textbf{(b)} Ata et al. [2024].
    \end{minipage}
    \caption{Optimal control policy for the EWF in the Pesic-Williams example with $\tilde{h}  = (1,2,3) $. The slanted arrows are $45$ degrees but appear otherwise due to our choice of scales.}
    \label{fig:pesic-EWF}
\end{figure}

Figure \ref{fig:pesic-EWF} admits the following interpretation. First, server 3 devotes all of its capacity to buffer 3, except when $W$ reaches the horizontal axis (that is, except when buffer 3 is empty), in which case server 3 idles. Second, server 1 devotes its full capacity to buffer 1 except when $W$ falls in the red region (including the vertical axis), in which case its full capacity  is devoted to activity 5 (the nonbasic activity). This means that, as predicted by \cite{pesic2016dynamic}, the nonbasic activity is used in the interior of the quadrant, and the optimal policy involves the curved free boundary shown in Figure \ref{fig:pesic-EWF}. Finally, because $\tilde h_1 <\tilde h_2$ for the case under discussion, the EWF solution has $Z_2(\cdot) \equiv 0$, which one interprets to mean that server 2 gives priority to buffer 2 at all times.

\subsection{Three-station example} 
\label{sec:ThreeStation}

Let us consider again the three-station network pictured in Figure \ref{fig:three-station}, assuming that jobs of types A, B, and C arrive to the system according to independent Poisson processes with average arrival rates 0.5, 0.25 and 0.25 jobs per hour, respectively. Also, class $k$ jobs have exponentially distributed service times with mean $m_k \, (k=1, \ldots ,8)$, where 
\begin{equation*}
( m_1, \ldots, m_8 ) = (2,3,1,1,1,1,2,1);
\end{equation*}
of course, for each class $k=1,\ldots,8$ the mean service rate is $\mu_k=1/m_k$. Recall from Subsection \ref{sec:InputControl} that we define for this example four different ``input activities" numbered 9, 10, 11 and 12, performed by three fictional ``input servers" numbered 4, 5 and 6. From the definitions of those activities given earlier, plus the hypothesized arrival rates for jobs of types A, B and C, we then have the following mean service rates:
\begin{equation}
\mu_9=0.5, \,\, \mu_{10}=0.5, \,\, \mu_{11}=0.25 \quad \text{and} \quad \mu_{12}=0.25.
\end{equation}
This means, for example, the following: if the fictional input server 6 devotes its full capacity to activity 12 (creation of class 6 jobs), that will generate arrivals into buffer 6 according to a Poisson process with average arrival rate 0.25 per time unit. Because all arrivals from the external world are generated by endogenous ``input activities" in this example, its associated vector of (uncontrollable) exogenous arrival rates is defined to be $\lambda=0$.

As in the criss-cross and parallel-server examples discussed earlier, the interest rate for discounting is taken to be $ \rho = 0.01$, and to formulate an approximating Brownian control problem, we take our scaling parameter to be $r = 400$. Thus, the interest rate used in our BCP is again $\gamma = r \rho =4$. \citet{harrison1996bigstep} and \citet{harrison1997dynamic} derived the initial BCP for this example, whose data include the matrices
\begin{equation}
    R =\left[ {\begin{array} {rrrrrrrrrrrr}
   1/2 &  &  &  & & & & & -1/2 &&& \\
   -1/2 & 1/3 &  & & & & & &&&& \\
    &  & 1 & & & & & & & -1/2 && \\
    & &  & 1 &  &  & & && & -1/4& \\
    & &  & -1& 1 & & & &&&&\\
    & & &       &      & 1 &  & &&&& -1/4 \\
    & & &       &     & -1& 1/2&  &&&& \\
    & & &       &      &      & -1/2 &1  &&&& \\
  \end{array} } \right]
  \label{eqn:Rvalues-3-station}
\end{equation}
and
\begin{equation}
    K =A =\left[ {\begin{array} {rrrrrrrrrrrr}
   1 &  &  &  & & &1 & &  &&& \\
    & 1 &  & & 1 & & & &&&& \\
    &  & 1 & 1& & 1& &1 & &  && \\
    & &  &  &  &  & & &1& 1& & \\
    & &  & & & & & &&&1&\\
    & & &       &      &  &  & &&&& 1 \\ 
  \end{array} } \right].
  \label{eqn:Avalues-3-station}
\end{equation}

In Appendix \ref{Nominal} the following 12-vector of nominal allocation rates is proposed for the three-station example:
\begin{equation*}
    \beta=(0.5,0.75,0.25,0.25,0.25,0.25,0.5,0.25,0.5,0.5,1,1).
\end{equation*}
The last four components of $\beta$ reflect a nominal policy of accepting all external arrivals, routing half of the type A arrivals through buffer 1 and the other half through buffer 3. Thus, under the nominal processing plan embodied in $\beta$, there will be external arrivals into buffers 1, 3, 4 and 6, each at average rate $0.25$ per time unit, but no external arrivals into buffers 2, 5, 7 or 8. Given the deterministic routing pictured in Figure \ref{fig:three-station}, the average flow rate through each of the model's eight buffers will then be 0.25. The first eight components of $\beta$ dictate that the capacities of the three original ``processing servers" be divided among their constituent buffers in proportions consistent with those flow rates.

The eight-dimensional Brownian motion $X$ underlying the initial BCP for this example has drift vector $\zeta=0$, and its covariance matrix is
\begin{equation*}
    \Gamma =\left[ {\begin{array} {rrrrrrrr}
   1/2 & -1/4 &  &  & & & &  \\
   -1/4 & 1/2 &  & & & & &  \\
    &  & 1/2 & & & & &  \\
    & &  & 1/2 & -1/4 &  & & \\
    & &  & -1/4 & 1/2 & & & \\
    & & &       &      & 1/2 & -1/4  & \\
    & & &       &      & -1/4 & 1/2 & -1/4  \\
    & & &       &      &      & -1/4 & 1/2 \\
  \end{array} } \right].
\end{equation*}
The model's economic parameters (that is, its holding cost rates and idleness cost rates) will be specified below.

\citet{harrison1996bigstep} and \citet{harrison1997dynamic} also derived
the EWF formulation for the three-station queueing network. It is a singular control problem with a two-dimensional state vector $W(t)=MZ(t)$ for $ t \geq 0$, where $Z$ is the eight-dimensional state vector for the initial BCP and 
\begin{equation*}
M=\left[ 
\begin{array}{cccccccc}
2 & 0 & 2 & 2 & 0 & 6 & 4 & 2 \\ 
3 & 3 & 3 & 4 & 1 & 6 & 3 & 3%
\end{array}%
\right] .
\end{equation*}
\begin{figure}[tbp]
\centering
\includegraphics[width=6in]{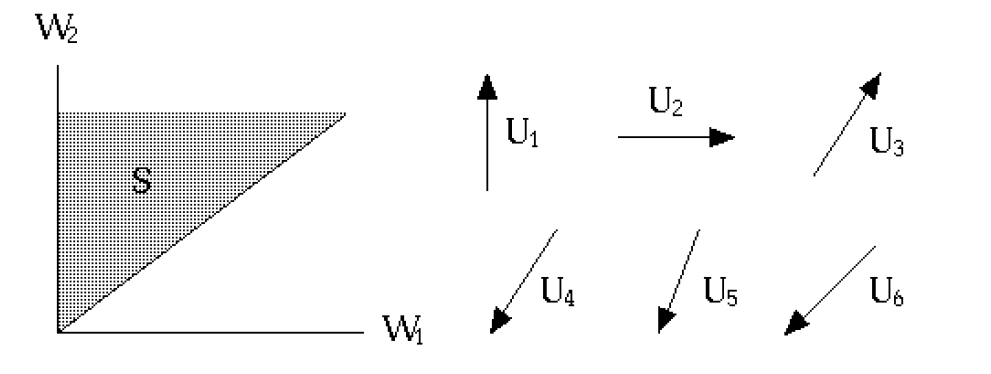}
\caption{The set $S$ and available directions of control (copied from 
\citet{harrison1996bigstep}).}
\label{fig:bigstep:control}
\end{figure}
Earlier in this section we saw that the workload dimension
of the criss-cross network (that is, the dimension of its EWF) equals the
number of that system's service stations, but in the current example the
system manager's dynamic routing capability reduces the workload dimension
from three (the number of service stations) to two. %
\citet{harrison1997dynamic} interpret components of $W$ as scaled workloads
for two overlapping ``server pools," one composed of servers 1 and 3, the
other composed of servers 2 and 3. The state space for $W$ is not the two-dimensional orthant, but rather the wedge
\begin{equation*}
S = \{w \in \mathbb{R}_+^2 : w = Mz, z \geq 0\},
\end{equation*}
which is pictured in Figure \ref{fig:bigstep:control}. The main system equation for the EWF is $W(t)=\eta(t)+GU(t)$ for $ t \geq 0$, where $\eta$ is a two-dimensional Brownian motion with zero drift and covariance matrix
\begin{equation*}
\Sigma=\left[ 
\begin{array}{cc}
18 & 20  \\ 
20 & 29 
\end{array}%
\right],
\end{equation*}
$U$ is a six-dimensional control with all components non-decreasing, and
\begin{equation*}
G=\left[ 
\begin{array}{cccccc}
1 & 0 & 2 & -1 & -1/2 & -3/2 \\ 
0 & 1 & 3 & -3/2 & -1 & -3/2%
\end{array}%
\right] .
\end{equation*}
The six columns of $G$ are the directions of displacement associated with the control components $U_1,\ldots,U_6$, and those directions are also pictured in Figure \ref{fig:bigstep:control}.

Control modes 1 through 3 correspond to idling servers 1 through 3, while 
control modes 4, 5 and 6 correspond to rejecting arrivals of types A, B and
C, respectively. The associated control cost vector $c$ has the form $c =
(0, 0, 0, c_4, c_5, c_6)$, where $c_4, c_5 \text{ and } c_6$ are
strictly positive. The three zero
components of $c$ reflect an assumption that there is no direct cost of
idling servers, whereas the positive values of $c_4, c_5 \text{ and } c_6$
represent the costs (expressed in units like dollars per minute) of ``turning off"
arrivals of types A, B and C, respectively; one can think of those ``costs"
either as direct penalties or as opportunity costs for business foregone. It
is noteworthy that, because the third column of $G$ can be written as a
positive linear combination of the first and second columns, and
furthermore, control modes 1, 2 and 3 are all costless (that is, $%
c_1=c_2=c_3=0$), control mode 3 can actually be eliminated from the problem
formulation. Equivalently, one can set $U_3(t)=0$ for all $t \geq 0$ without
loss of optimality, which is consistent with the optimal policy described below.

The remaining data for the EWF of the three-station example are its holding cost rates and idleness cost rates. Here we directly specify that the \textit{scaled} cost rates for use in the BCP and its EWF are as follows:
\begin{eqnarray*}
    \tilde h = (6,3,6,6,1,12,7,6) \text{\quad and \quad} \tilde c = (0, 0, 0, 0.495, 0.3, 0.675).
\end{eqnarray*}
Of course, the unscaled cost rates that give rise to these values can be determined from the formulas for $\tilde h$ and $\tilde c$ in \eqref{eq:scaledcoeffs}, plus the value $r=400$ of the scaling parameter. In particular, the unscaled idleness cost rates for the fictional servers numbered 4, 5 and 6 are $(c_4,c_5,c_6)=(198,120,270)$.

An important point to note is the following: each positive component of $c$ represents a \textit{cost per unit of idle time} for the corresponding input server, not a cost per arrival rejected, so the average arrival rate for the input stream must be factored into its computation. Positive components of $\tilde c$ have the same meaning, but with both cost and time expressed in scaled terms.

We now define the holding cost function $\tilde h(\cdot)$ for the three-station EWF as in our previous treatment of the criss-cross and Pesic-Williams examples, that is, 
\begin{equation}
\tilde h(w)=\min\{\tilde h \cdot z:Mz=w,z\geq 0 \},\quad w\in S.
\label{eqn:3station-running-cost} 
\end{equation}
As before, one can interpret or animate \eqref{eqn:3station-running-cost} as follows: the system manager controls the workload process $W$ directly (in this case, through both idleness decisions and input control decisions), and then, at each time $t$, can take the queue length vector $Z(t)$ to be any $z \in \mathbb{R}_+^8$ satisfying $Mz=W(t)$; granted that latitude, the system manager chooses the $z$ value with the minimum associated total holding cost.

\begin{figure}[htbp]
\centering
\begin{subfigure}[b]{0.4\textwidth}
\centering
\includegraphics[width=\textwidth]{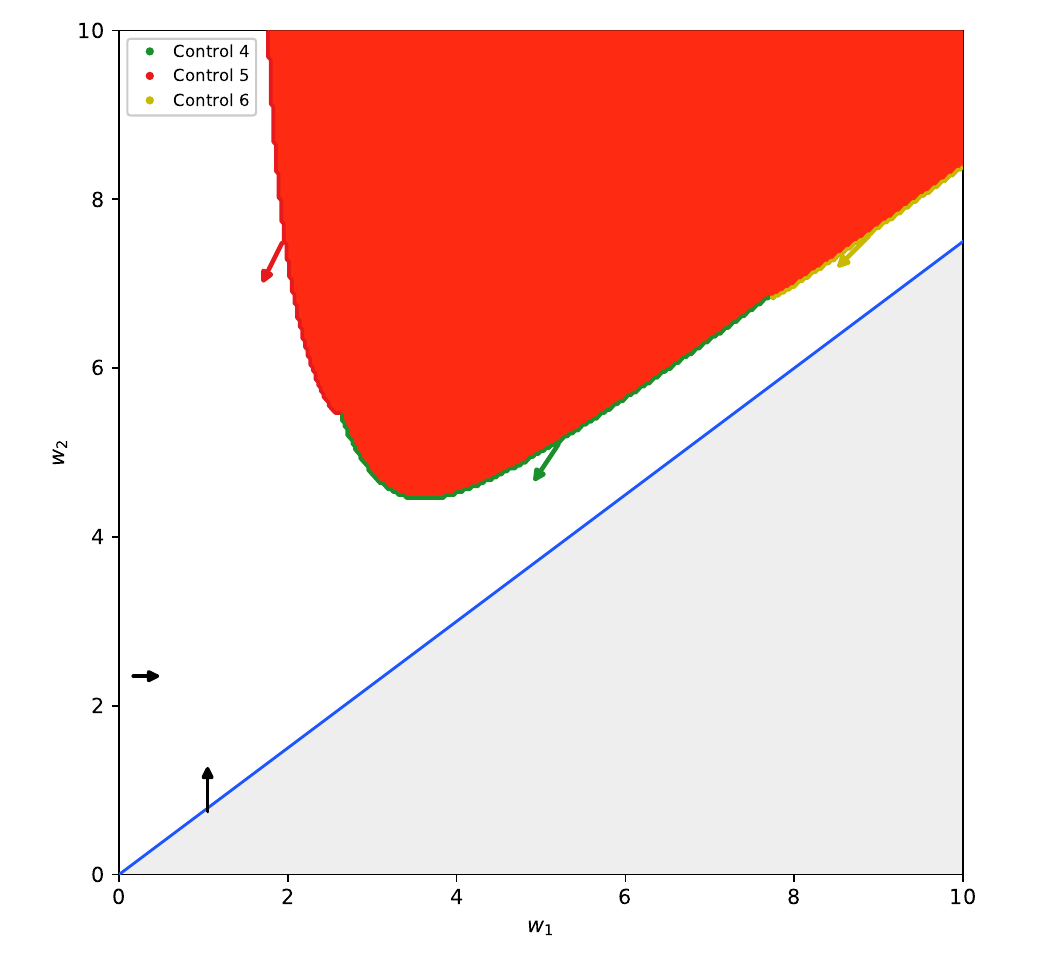}
\caption{KM policy}
\end{subfigure}
\hfill
\begin{subfigure}[b]{0.4\textwidth}
\centering
\includegraphics[width=\textwidth]{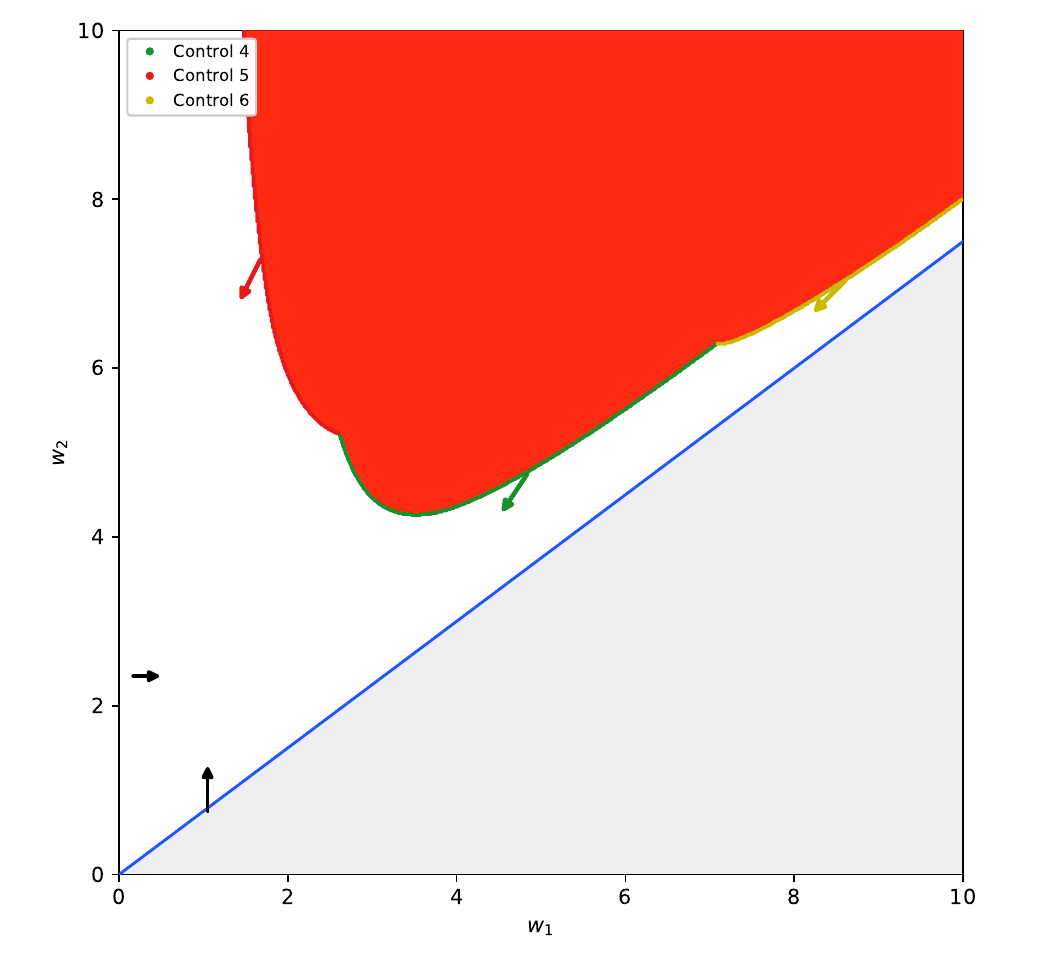}
\caption{AHS policy}
\end{subfigure}
\caption{Optimal control policy for the three-station EWF. }
\label{figure:kusher}
\end{figure}

For the problem parameters specified above, we have solved the EWF using the computational method developed by \cite{kushner1991numerical}, and then using the method proposed by \cite{ata2024singular}. The solutions obtained by the two methods are displayed in Figure \ref{figure:kusher}, labeled there as the KM policy and the AHS policy. The two policies are very similar, each one using only five of the six modes of control, which one interprets as follows: by using the alternate routing capability mentioned earlier, and idling servers 1 and 2 when necessary, the system manager is able to never idle server 3 (that is, can set $U_3(\cdot) \equiv 0$). The black arrows in Figure \ref{figure:kusher} show that servers 1 and 2 are idled only when there is no potential work for that server to do anywhere in the system, and the three colored arrows show that new arrivals of each type A, B or C are rejected in a particular part of the state space. As in the Pesic-Williams example, the two methods also produce nearly identical optimal objective values, differing by no more than 0.25\%.

Finally, but not shown in the figure, only two components of the eight-vector $Z(t)$ will be positive at a typical time $t$, because the minimization problem \eqref{eqn:3station-running-cost} is a linear program with two equality constraints, so basic solutions of that LP problem have only two positive variables. Of course, just \textit{which} two scaled queue length variables $z_i$ (that is, which scaled buffer content variables) take positive values in \eqref{eqn:3station-running-cost} depends on the current workload vector $w$: the white area in Figure \ref{figure:kusher} divides into polyhedral regions that each correspond to a different combination of two non-empty buffers.

How is this EWF solution to be implemented in our original problem context? Its input control recommendations (that is, when to reject external arrivals of each type) are straightforward, as are its instructions with regard to server idleness, but as in the criss-cross example discussed earlier in this subsection, the recommendations with regard to buffer content levels cannot be taken literally. How are we to route type A arrivals, and to prioritize the job classes served at each station, so as to maintain approximately the recommended buffer content levels, and to keep server 3 working continuously as the system evolves over time?

By analogy with the criss-cross example, it seems plausible that the buffer content values of zero that appear in our EWF solution should each be interpreted to mean  ``small" (in comparison with our spatial scaling factor $\sqrt{r}$), and that the desired implementation can be achieved through a system of dynamic priorities based on threshold parameters. However, translating that vague philosophical stance into a detailed policy prescription, not just for this example but in general, is a problem of head-spinning complexity, and so we now consider a radically different approach.

\section{A modified BCP with bounded drift rates as controls}
\label{sec:drift1}
As noted previously, our initial BCP (Subsection \ref{sec:raw}) features an $n$-dimensional control process $Y$ whose individual components $Y_j$ are allowed to have unbounded variation, and that leads to the mathematically equivalent EWF described in Subsection \ref{sec:EWF}. For a less extreme idealization of the network control problem originally formulated in Section \ref{sec:general}, we now modify the initial BCP by restricting attention to controls $Y$ of a form specified immediately below. 

The state descriptor in our modified BCP, as in the initial BCP, is an $m$-dimensional process $Z$ taking values in the orthant $\mathbb{R}_+^m$, but now admissible controls must have the form
\begin{equation}
dY(t) = \theta(t)\,dt + Q\,dL(t),
\label{eq:formofY}
\end{equation}
where: $\theta=\{\theta(t), t \geq 0\}$ is a right-continuous, $n$-dimensional drift control process, adapted to the Brownian motion $X$, to be chosen by the system manager subject to the upper bound constraint \eqref{eq:bdd_density} below; $L=\{L(t), t \geq 0\}$ is an $m$-dimensional, non-decreasing \textit{boundary process} (that is, components of $L$ increase only when $Z$ is on the boundary of the orthant) that will be defined in terms of $\theta$ via \eqref{eq:pushing} and \eqref{eq:complementarity} below; and $Q$ is an $n \times m$ matrix whose structure will be specified in Subsection \ref{sec:drift2} to follow. (The letter $Q$ was used earlier in Section \ref{sec:general} to denote an $m$-dimensional buffer contents process, or queue length process, but that notation will not be needed in the remainder of the paper.) Given that controls have the form \eqref{eq:formofY}, the main system equation \eqref{eq:BCP1} in our initial BCP takes the following form:
\begin{equation}
Z(t) = Z(0) + X(t) + \zeta t + R\left[\int _0^t \theta(s)\,ds + QL(t)\,\right] \geq 0, \quad t \geq 0.
\label{eq:BCP1new}
\end{equation}

Given a drift control $\theta$, the boundary process $L$ is defined by the combination of \eqref{eq:BCP1new} and the following relationships:
\begin{equation}
L_i(\cdot) \text{ is continuous and non-decreasing with } L_i(0)=0 \text{ } (i=1,\ldots,m),
\label{eq:pushing}
\end{equation}
and
\begin{equation}
L_i(\cdot) \text{ increases only at times $t$ when } Z_i(t)=0 \text{ } (i=1,\ldots,m).
\label{eq:complementarity}
\end{equation}
The effect of $L(\cdot)$ is to impose a lower reflecting barrier at zero on each component of $Z$. More precisely, the \textit{direction of reflection} from the boundary surface $Z_i=0$ is given by the $i${th} column of the $m \times m$ matrix 
\begin{equation}
H:=RQ,
\label{eq:defineG}
\end{equation}
and $Q$ will be chosen so that 
\begin{equation}
H \text{ is a completely-S matrix}, 
\label{eq:BCP2new}
\end{equation}
as defined, for example, in the book by \citet{cottle2009linear}. Given this restriction, $Z$ is well defined as a semimartingale reflected Brownian motion (SRBM), cf. \citet{taylor1993existence}. Its state-dependent drift depends on the system manager's control policy $\theta$, and it has the boundary behavior described above.

One can think of $\theta$ as a policy chosen by the system manager for controlling $Z$ on the interior of the orthant, whereas $L$ and $Q$ together constitute a pre-specified policy for control at the boundary. Elements of $Q$ are artificial parameters chosen to promote computational efficiency, but they can also be viewed as specifying a \textit{baseline control policy} that underlies our computational procedure; see below for further discussion. The crucial added restriction in our modified BCP formulation is that each component of the system manager's drift control process must be absolutely continuous with uniformly bounded density. To be specific, we require that
\begin{equation}
\theta_j(\cdot)\leq a_j \text{ for all } j=1,\ldots,n,
\label{eq:bdd_density}
\end{equation}
where $a_1,\ldots, a_n >0$ are upper bound parameters chosen at the discretion of the analyst; an argument will be given shortly to suggest that values of order $\sqrt r$ are appropriate for faithful representation of our original network control problem. Readers can easily check that the upper bounds \eqref{eq:bdd_density}, the policy constraint \eqref{eq:newBCP3} imposed below, and the definition of the matrix $K$ together imply that components of $\theta(\cdot)$ are uniformly bounded \textit{below} as well.

To review, our specification of a modified BCP includes all of the relationships \eqref{eq:BCP1new}, \eqref{eq:pushing}, \eqref{eq:complementarity} and \eqref{eq:bdd_density}, where $X$ is an $m$-dimensional Brownian motion with zero drift and the covariance matrix $\Gamma$ defined via \eqref{eq:DefineX}, $\zeta$ is the nominal drift vector defined via \eqref{eq:heavy_traffic}, $\theta$ is an $n$-dimensional drift control process adapted to $X$, and $Q$ and $a_1,\ldots, a_n$ are tuning parameters chosen to facilitate computations.

To complete the problem specification, we need to specialize the capacity constraints \eqref{eq:BCP3} and the objective function \eqref{eq:Brownianobjective} to controls of the form \eqref{eq:formofY} considered here. With that goal in mind, we next require that $Q$ be chosen so that
\begin{equation}
KQ\geq 0,
\label{eq:newBCP6}
\end{equation}
where $K$ is the $(p+n-b) \times n$ matrix defined via \eqref{eq:define_K}. The inequality \eqref{eq:newBCP6} ensures that each component of the vector process $KQL(\cdot)$ is non-decreasing, and with that restriction in place, \eqref{eq:BCP3} reduces to the following constraint on the system manager's drift control process $\theta(\cdot)$: 
\begin{equation}
K\theta(\cdot)\geq 0.
\label{eq:newBCP3}
\end{equation}
Similarly, given \eqref{eq:formofY} and the definition $U(t)=KY(t)$, the differential $dU(t)$ in our BCP objective \eqref{eq:Brownianobjective} can be re-expressed as $K\theta(t) \, dt+ KQ\,dL(t)$. Thus the objective in our modified BCP is the following:
\begin{equation}
 \text{minimize  } \mathbb{E} \left \{ \int_{0}^{\infty} e^{- \gamma t} [\, \tilde h\,' Z(t) \, dt + \tilde c\,' K\theta(t) \, dt + \tilde c\,'KQ\,dL(t) \,] \right \}.
\label{eq:newBrownianobjective}
\end{equation}
(Here primes are used to denote transposes of the column vectors $\tilde h$ and $\tilde c$, because the inner product notation used earlier in \eqref{eq:Brownianobjective} might cause confusion in the current context.) For each of the examples discussed later in Section \ref{sec:Tests}, the matrix $Q$ will be chosen so that 
\begin{equation}
\tilde c\,'KQ=0,
\label{eq:nobdrycost}
\end{equation}
which causes the final term on the right side of \eqref{eq:newBrownianobjective} to vanish, regardless of how $\theta(\cdot)$ may be chosen. In fact, for virtually all of the applications we currently envision, it will be natural to choose a matrix $Q$ that satisfies the identity \eqref{eq:nobdrycost}. As our examples will show, \eqref{eq:nobdrycost} holds when (a) the only positive elements of the cost rate vector $\tilde c$ are penalty rates associated with rejecting arrivals of various classes, and (b) the control actions embodied in the matrix $Q$ do not involve any such rejections. 

Hereafter, the stochastic control problem \eqref{eq:BCP1new}-\eqref{eq:newBrownianobjective} will be referred to as our \textit{modified BCP}. This is a problem of the form considered in our earlier paper \citet{ata2023drift}, so it can be solved via the ``deep learning" method developed there; see Section \ref{sec:Tests} below for illustrative applications. Our computational method takes as inputs the system parameters $\Gamma$, $\zeta$ and $R$, the interest rate $\gamma>0$, and the scaled cost rate vectors $\tilde h$ and $\tilde c$, plus the tuning parameters $Q$ and $a$; it returns as outputs the problem's optimal value function $V(\cdot)$ and its associated gradient function $\nabla V(\cdot)$. In Section \ref{sec:policy} below, we will construct a continuous-review queueing control policy based on $\nabla V(\cdot)$.

Because components of $Y$ correspond to scaled deviation controls (see Section \ref{sec:Centering}), the crucial policy constraint \eqref{eq:bdd_density} is interpreted to mean that deviations from the nominal processing plan must occur at bounded rates, at least away from the boundary of the state space. This eliminates the potential for instantaneous state adjustments by the system manager, and hence eliminates the state space collapse illustrated in Subsection \ref{sec:EWF}.
\newline

\noindent \textbf{Some guidance on how to set the upper bound parameters $a_1,\ldots, a_n$.} The capacity constraints \eqref{eq:capacity-constraint} in our original formulation of the network control problem require that the vector process $T(\cdot)$ of cumulative capacity allocations be absolutely continuous. It follows that the process $\tilde Y(\cdot)$ of scaled deviation controls must be absolutely continuous as well. That is, $d\tilde{Y}(t)=\tilde \theta(t) \, dt$ for some density process $\tilde \theta(\cdot)$, and by combining various relationships that appear in Section 3, one finds that components of $\tilde \theta(\cdot)$ must satisfy 
\begin{equation}
-\sqrt r (1-\beta_j) \leq \tilde \theta_j(t) \leq \sqrt r \beta_j \quad \text{for  } t \geq 0 \text{   and   } j=1,\ldots,n.
\label{eq:more_bounds}
\end{equation}
Because the control process $Y(\cdot)$ in our modified BCP serves as a surrogate for $\tilde Y(\cdot)$, its density $\theta(\cdot)$ may be viewed as a surrogate for $\tilde \theta(\cdot)$, so \eqref{eq:more_bounds} suggests that the upper bound parameters $a_1,\ldots, a_n$ in \eqref{eq:bdd_density} should be of order $\sqrt r$ to faithfully reflect our original problem context.

\subsection{The matrix $Q$ and its interpretation}
\label{sec:drift2}
To ensure that the state space constraint $Z(\cdot) \geq 0$ can be met, despite the upper bound constraint \eqref{eq:bdd_density} on drift rate controls, we have included the boundary term $Q\,dL(t)$ in the general form \eqref{eq:formofY} of an admissible control for our modified BCP.  However, the resulting boundary behavior is not just a mathematical contrivance: each component $L_i(t)$ can be interpreted as (a scaled version of) the cumulative amount of time that buffer $i$ has been empty over the interval $[0,t]$, and then elements of $Q$ specify deviations from nominal activity rates that are implemented in conjunction with such emptiness; positive values specify \textit{decreases} from nominal activity rates (or equivalently, from nominal capacity allocations), and negative values specify increases.

For a more complete explanation of that view, let us first recall that (a) the total capacity of each server is 1 by convention (see Section \ref{sec:general}), and (b) the $n$-vector $\beta$ of nominal activity rates satisfies $A\beta=e$ by assumption (see Section \ref{sec:Centering}); the latter equality is interpreted to mean that the nominal processing plan exhausts the capacities of all servers. Combining these two elements of the model, one sees that each of the nominal activity rates $\beta_j$ can be interpreted as a capacity allocation: specifically, $\beta_j$ is the fraction of the capacity of server $s(j)$ that is devoted to activity $j$ under the nominal processing plan. We now specify that, for each $i=1,\ldots,m$,
\begin{equation}
    Q_{ji} \text{  }
\begin{cases} 
= \beta_j  & \text{if } b(j) = i, \\
= 0         & \text{if } b(j) = 0, \\
\leq 0     & \text{otherwise}.
\end{cases}
\label{eq:defineQ}
\end{equation}
In words, \eqref{eq:defineQ} specifies the following deviations from nominal activity rates when buffer $i$ is empty ($i=1,\ldots,m$). First, any activity $j$ that serves buffer $i$ is discontinued, which releases a fraction $\beta_j$ of the capacity of server $s(j)$ for possible alternative uses. Second, no adjustments are made to the nominal activity rates for input activities. Third, all other activity levels are increased either by positive amounts, subject to the capacity constraint \eqref{eq:capcon} below, or not at all. 

An important element of our modified BCP formulation (Section \ref{sec:drift1}) is the policy constraint \eqref{eq:newBCP6}, requiring that $KQ \geq 0$. Reviewing our definition \eqref{eq:define_K} of the matrix $K$, one sees that \eqref{eq:newBCP6} restricts the choice of $Q$ in two ways. First, it requires that
\begin{equation}
AQ\geq0. 
\label{eq:capcon}
\end{equation}
This means that, for each $i=1,\ldots,m$, the activity level deviations embodied in the $i$th column of $Q$ are respectful of capacity constraints. That is, the activity level \textit{increases} that are implemented when $Z_i=0$ require no more added capacity than the amounts released by discontinued basic activities. The second requirement embodied in  \eqref{eq:newBCP6} is that, on each of the boundary surfaces $Z_i=0 \text{ } (i=1,\ldots,m)$, nonbasic activities must be introduced at nonnegative rates, if at all. That second requirement has already been included in our specification \eqref{eq:defineQ}. Our final constraint on the choice of $Q$ is \eqref{eq:BCP2new}, requiring that the $m \times m$ matrix $H:=RQ$ be completely-S. As stated previously, this restriction is essential, ensuring that the process $Z$ in \eqref{eq:BCP1new} is well defined.

In the remainder of this subsection, we discuss the $Q$-matrix choices that were used in our computational studies of the three examples introduced earlier in Section \ref{sec:general}. For purposes of this discussion, the notation S1 is used to mean ``server 1," B1 to mean ``buffer 1," A1 to mean ``activity 1," and similarly for other servers, buffers and activities.
\newline
\newline
\textbf{Criss-cross example.} To formulate the modified BCP for the criss-cross network, we set
\begin{equation}
   Q=
  \left[ {\begin{array} {ccc}
   0.5 & - 0.5\kappa & 0 \\
   - 0.5 \kappa & 0.5 & 0 \\
   0 & 0 & 1 \\
  \end{array} } \right],
  \label{eq:defineQ-matrix}
\end{equation}
where $\kappa \in (0,1)$. For an interpretation of this choice, first consider the meaning of \eqref{eq:defineQ-matrix} if one were to take $\kappa =1$. The nominal processing plan (see Appendix \ref{Nominal}) dictates that S1 divide its capacity equally between A1 and A2 (that is, between serving B1 and serving B2), while S2 devotes all of its capacity to A3. If one were to take $\kappa=1$, then the first column of $Q$ would specify the following whenever B1 is empty (that is, whenever $Z_1=0$): S1 will redirect half of its capacity from A1 to A2, thus devoting \textit{all} of its capacity to serving B2. Similarly, the second column of $Q$ would specify that, whenever B2 is empty (that is, whenever $Z_2=0$), S1 will devote all of its capacity to serving B1. Finally, the third column of $Q$ specifies that, whenever B3 is empty (that is, whenever $Z_3=0$), S2 simply idles.

By choosing a value $\kappa<1$, we are specifying the following for this particular example. First, when A1 must be suspended because B1 is empty, the S1 capacity that is nominally allocated to A1 cannot all be reallocated; specifically, a fraction $(1-\kappa)$ of that nominal allocation must be left unused, contributing to the cumulative idleness of S1, while the remaining fraction $\kappa$ is reallocated to A2. A similar statement holds about the reallocation of S1 capacity when B2 is empty (that is, when $Z_2=0$). With $Q$ specified via \eqref{eq:defineQ-matrix}, our definition \eqref{eq:defineG} of the matrix $H$ gives the following:
\begin{equation}
H = R \, Q =
  \left[ {\begin{array} {ccc}
   1 & - \kappa & 0 \\
   -\kappa & 1 & 0 \\
    \kappa & -1 & 1 \\
  \end{array} } \right].
  \label{eq:defineH-matrix}
\end{equation}
It will be shown in Appendix \ref{Verifying} that this matrix $H$ is completely-S for any choice of $\kappa \in (0,1)$, and readers can easily verify that \eqref{eq:defineQ-matrix} also meets the other restrictions imposed on $Q$ earlier in this section. In the numerical experiments reported later in Section \ref{sec:Tests}, we have used the value $\kappa = 0.99$, which means that one-half of one percent of S1 capacity remains unused if either B1 or B2 is empty. 

Using general language that applies equally well to all three of our examples, the condition $\kappa<1$ has the following interpretation: whenever a basic activity must be suspended or discontinued due to emptiness of its associated buffer, a fraction $(1-\kappa)$ of its nominal capacity allocation remains unused, contributing to the cumulative idleness of the associated server. This requirement ensures that the matrix $H=RQ$ is completely-S for each of our examples (see Appendix \ref{Verifying}); for our criss-cross and three-station examples, it is easy to show that $H$ is \textit{not} completely-S if one takes $\kappa=1$.
\newline
\newline
\textbf{Pesic-Williams parallel-server example.} For this three-server example we set 
\begin{equation}
   Q =
  \left[ {\begin{array} {ccc}
   1 & 0 & 0 \\
   0.5 & - 0.5 \kappa & 0 \\
   - 0.5 \kappa  & 0.5 & 0 \\
   0 & 0 & 1 \\
   - \kappa  & 0 & 0
  \end{array} } \right],
  \label{eq:QforPW}
\end{equation}
where $\kappa \in (0,1)$, and from that one obtains
\begin{equation}
   H = RQ =
  \left[ {\begin{array} {ccc}
   2 & - \kappa & 0 \\
   - \kappa  & 1 & 0 \\
   - \kappa & 0 & 1 \\
  \end{array} } \right].
  \label{eq:HforPW}
\end{equation}
As in the criss-cross example, the condition $\kappa<1$ ensures that $H$ is completely-S (see Appendix \ref{Verifying}), and \eqref{eq:QforPW} also meets the other restrictions imposed on $Q$ earlier in this section. For the numerical experiments reported later in Section \ref{sec:Tests}, we again set $ \kappa = 0.99$.

For an interpretation of \eqref{eq:QforPW}, first recall that the nominal processing plan calls for S1 to use all of its capacity serving B1 (this is designated as A1), while S2 divides its capacity equally between serving B1 and serving B2 (these are A2 and A3), and S3 devotes all of its capacity to serving B3 (designated as A4). The first column of $Q$ in \eqref{eq:QforPW} then specifies the following adjustments when B1 is empty (that is, when $Z_1=0$): S1 reallocates a fraction $\kappa$ of its capacity to serving B3 (this is designated as A5), while the remaining fraction $(1-\kappa)$ is unused; and second, a fraction $\kappa$ of the S2 capacity that is nominally allocated to serving B1 is reallocated to serving B2 (that is, reallocated to A3), while the remaining fraction $(1-\kappa)$ of that nominal allocation is unused. The second column of $Q$ specifies a similar adjustment in the use of S2 capacity when $Z_2=0$, and the third column specifies that S3 is idled when $Z_3=0$.
\newline
\newline
\textbf{Three-station example.} For this more complex network, $Q$ is a $12 \times 8$ matrix. We take its first 8 rows to be 
\begin{equation}
  Q_{8\times 8}=
  \left[ {\begin{array} {c}
  \begin{array} {rrrrrrrr}
   1/2 & &  &  & & & -\kappa/2 &  \\
    & 3/4 &  & &-\kappa/4 & & &  \\
    &  & 1/4 & -\kappa/12 & & -\kappa/12 & & -\kappa/12  \\
    &  & -\kappa/12 & 1/4 &  &-\kappa/12  & &-\kappa/12 \\
    &- 3 \kappa/4 &  &  & 1/4 & & & \\
    & & -\kappa/12& -\kappa/12      &      & 1/4 &  &-\kappa/12 \\
   -\kappa/2 & & &       &      &  & 1/2&   \\
    & & -\kappa/12&  -\kappa/12     &      &  -\kappa/12    & & 1/4 \\
  \end{array}\\[1mm]
  \end{array} } \right],
\label{eq:LastOne}
\end{equation}
where $\kappa \in (0,1)$; its last 4 rows are the $4\times 8$ matrix of zeros. That choice gives
\begin{equation}
   H = RQ =
  \left[ {\begin{array} {rrrrrrrr}
   1/4 & &  &  & & & -\kappa/4&  \\
   -1/4 & 1/4 &  & &-\kappa/12 & & \kappa/4 &  \\
    &  & 1/4 & -\kappa/12 & & -\kappa/12 & & -\kappa/12  \\
    &  & -\kappa/12 & 1/4 &  &-\kappa/12  & &-\kappa/12 \\
    &-3\kappa/4 & \kappa/12  &-1/4  & 1/4 &\kappa/12  & & \kappa/12 \\
    & & -\kappa/12& -\kappa/12      &      & 1/4 &  &-\kappa/12 \\
   -\kappa/4 & &\kappa/12  &  \kappa/12      &      & -1/4 & 1/4& \kappa/12   \\
   \kappa/4  & & -\kappa/12&  -\kappa/12     &     &  -\kappa/12    & -1/4& 1/4 \\
  \end{array} } \right].
\label{eq:ThirdH}
\end{equation}
This $8 \times 8$ matrix $H$ is completely-S (see Appendix \ref{Verifying}) for any choice of $\kappa \in (0,1)$, and again we have taken $ \kappa = 0.99$ for our computational study. Also, one can verify that our choice of $Q$ satisfies all the other restrictions that were imposed earlier in this section.

For an interpretation of \eqref{eq:LastOne}, let us consider the fourth column of $Q$, which specifies deviations from the nominal activity rates that are to be implemented when B4 is empty (that is, when $Z_4=0$). The server responsible for processing jobs from B4 is S3, and the nominal processing plan for this example (see Appendix \ref{Nominal}) calls for S3 to divide its capacity equally among its four constituent buffers, that is, to divide its capacity equally among A3, A4, A6 and A8. When $Z_4=0$, the one-fourth of S3 capacity that is nominally allocated to A4 must be reallocated: a fraction $\kappa$ of that nominal allocation is divided equally among the server's other three responsibilities, so the activity rates (or equivalently, the capacity allocations) for A3, A6 and A8 are all increased by $\kappa/12$; finally, a fraction $(1- \kappa)$ of the nominal allocation is left unused. Other columns of $Q$ are interpreted similarly, each of them specifying deviations from the nominal plan that are to be implemented when a particular buffer is empty. 

\subsection{Admissibility of the drift controls allowed}
\label{sec:convergence}
Denoting by $Q^i$ the $i$th column of $Q$, and rewriting \eqref{eq:formofY} in a slightly expanded form, we have the following: the controls $Y(\cdot)$ in our modified BCP are those of the form 
\begin{equation}
   dY(t)=\theta(t) \, dt + \sum_{i=1}^m Q^i \, dL_i(t),
\label{eq:control}
\end{equation} 
where $\theta(\cdot)$ and $Q$ satisfy restrictions stated earlier in this section. Those restrictions on $\theta(\cdot)$ and $Q$ include $K \theta(\cdot) \geq 0$ and $KQ^i \geq 0$ for each $i$, while $L_i(\cdot)$ is non-decreasing for each $i$ by definition. Thus $KY(\cdot)$ is non-decreasing as well, so each $Y(\cdot)$ of the form \eqref{eq:control} is an admissible control for our initial BCP (Subsection \ref{sec:raw}). 

\section{A network control policy based on the modified BCP solution}
\label{sec:policy}
To solve the modified BCP introduced in Section \ref{sec:drift1}, we restrict attention to stationary Markov controls. That is, attention will be restricted to drift controls of the form
\begin{equation}
    \theta (t) = u(Z(t)), \ \ t \geq 0,
    \label{eqn:markov-policy-defn}
\end{equation}
for some measurable function $u: \mathbb{R}_+^m \rightarrow \Theta$, where $\Theta$ is the action space for our control problem, namely,
\begin{equation}
    \Theta = \left\{ \theta \in \mathbb{R}^n : K \theta \geq 0 \, \text{ and }  \theta_j \leq a_j \, \text{ for all } j=1,\ldots,n \right\};
\label{eqn:defineTheta}
\end{equation}
here $a_1,\ldots, a_n$ are the upper bound parameters in \eqref{eq:bdd_density}. The function $u$ in \eqref{eqn:markov-policy-defn} will be called a \textit{policy}, and we denote by $\mathcal{U}$ the set of all such $u$. Also, we denote by $Z^u$ the controlled RBM under policy $u$, which is defined via (\ref{eq:BCP1new}) and (\ref{eqn:markov-policy-defn}). Next, assuming that $Q$ is chosen to satisfy \eqref{eq:nobdrycost}, let $V^u(\cdot)$ be the value function for policy $u$, meaning that
\begin{eqnarray*}
    V^u(z) = \mathbb{E}_z\left[ \int_0^\infty  e^{- \gamma t} \left( \tilde{h}^{'} Z^u(t) + \, \tilde{c}^{\,'}K u(Z^u(t)) \right) dt\right], \, z \in \mathbb{R}_+^m,
\end{eqnarray*}
where $\mathbb{E}_z(\cdot)$ denotes as usual a conditional expectation given that $Z(0) = z$. Finally, we define the optimal value function
\begin{equation*}
    V(z) = \min_{u \in \mathcal{U}} V^u(z) \ \text{ for each } z \in \mathbb{R}_+^m.
\end{equation*}
 
\noindent \textbf{Hamilton-Jacobi-Bellman equation for the modified BCP.} For an arbitrary $C^2$ function $f : \mathbb{R}^m \rightarrow \mathbb{R}$, let $\nabla f: \mathbb{R}^m \rightarrow \mathbb{R}^m$ be the corresponding gradient function. Also, we define a second-order differential operator $\mathcal{L}$ via
\begin{eqnarray*}
    \mathcal{L}f = \frac{1}{2} \sum_{i=1}^m \sum_{j=1}^m \Gamma_{ij} \frac{\partial^2 f}{\partial z_i \, \partial z_j} ,
\end{eqnarray*}
and a first-order operator $\mathcal{D} = ( \mathcal{D}_1, \ldots, \mathcal{D}_m)$ via
\begin{equation*}
    \mathcal{D}f =H^{'} \nabla f;
\end{equation*}
here $\Gamma$ is the covariance matrix for the modified BCP being solved, and $H$ is defined via \eqref{eq:defineG}. Thus, $\mathcal{D}_if(\cdot)$ is the directional derivative of $f$ in the direction of reflection on the boundary surface $\{z \in \mathbb{R}_+^m : z_i =0\}$. The optimal value function $V(\cdot)$ for our modified BCP satisfies the HJB equation
\begin{equation}
    \mathcal{L}V(z) + \min_{\theta \in \Theta} \left\{ \nabla V(z) \cdot R\theta + \tilde{c}^{\,'}K \theta \right\} 
    + \zeta \cdot \nabla V(z) + \tilde{h} \cdot z = \gamma V(z), \ z \in \mathbb{R}_+^m,
\end{equation}
with boundary conditions
\begin{equation}
    \mathcal{D}_iV(z) = 0 \ \text{ if } \  z_i =0 \ (i=1,\ldots,m),
\end{equation}
and given a solution, an optimal policy function $u^*(\cdot)$ can be computed using $\nabla V(\cdot)$ as usual:
\begin{equation}
   u^*(z) =  \text{argmin}_{\, \theta \in \Theta} \left\{ \nabla V(z) \cdot R\theta + \tilde{c}^{\,'}K \theta \right\}, \ z\geq 0.
   \label{eqn:optimal-drift-policy}
\end{equation}

\noindent \textbf{Characterization of the optimal policy.} Let $R^j$ and $K^j$ denote the $j^{th}$ column of $R$ and the $j^{th}$ column of $K$, respectively. Also, let us denote by $\mathcal{A}(k)$ the set of activities that are conducted by server $k$, that is,
\begin{equation*}
\mathcal{A}(k)= \left\{j \in \{1,\ldots,n\}:A_{kj}=1 \right\} \quad \text{for  } k=1,\ldots,p.
\end{equation*}
We can then re-express the linear program on the right side of (\ref{eqn:optimal-drift-policy}) as follows: choose $\theta \in \mathbb{R}^n$ to 
\begin{equation}
 \text{minimize } \ \sum_{j=1}^n \left( \nabla V(z) \cdot R^j + \tilde{c} \cdot K^j \right) \theta_j 
    \label{eqn:LP-objective}
\end{equation}
subject to
\begin{equation}
\sum_{j \in \mathcal{A}(k)} \theta_j \geq 0 \quad \text{for all  } k=1,\ldots,p,  
\label{eqn:LP-sec-constraint}
\end{equation}
\begin{equation}
\theta_j \leq 0 \quad \text{for all} \quad j=1,\ldots,n \text{  such that  }  \beta_j=0,
\label{eqn:LP-sec-constraint2}
\end{equation}
and
\begin{equation}
\theta_j \leq  a_j \quad \text{for all  } j=1,\ldots,n.
\label{UB}
\end{equation}
(Together, the constraints \eqref{eqn:LP-sec-constraint} and \eqref{eqn:LP-sec-constraint2} are equivalent to $K\theta\geq0$.)
Crucially, the linear program \eqref{eqn:LP-objective}-\eqref{UB} can be decomposed into $p$ separate problems, one for each server, as follows. For each $k= 1, \ldots,p$, let $\Theta(k)$ be the set of drift components $\left\{\theta_j, j \in \mathcal{A}(k) \right\}$ that satisfy \eqref{eqn:LP-sec-constraint2} and \eqref{UB} for all $j \in \mathcal{A}(k)$ and also satisfy \eqref{eqn:LP-sec-constraint}. Then \eqref{eqn:LP-objective}-\eqref{UB} can be restated as follows: for each $k= 1, \ldots,p$, choose $\left\{\theta_j, j \in \mathcal{A}(k) \right\} \in \Theta(k)$ so as to
\begin{equation}
\text{minimize } \ \sum_{j \in \mathcal{A}(k)} \left( \nabla V(z) \cdot R^j + \tilde{c} \cdot K^j \right) \theta_j.
\label{eqn:decomposed-LP-objective}
\end{equation}
To solve this problem for a given state $z \in \mathbb{R}_+^m$, let us define the index function
\begin{equation}
    \pi_j(z) = \nabla V(z) \cdot R^j + \tilde{c} \cdot K^j, \ \ j=1,\ldots,n.
\label{eq:DefinePi}
\end{equation}
Components of the $m$-vector $R^j$ are the rates at which one \textit{increases} components of the state vector $z$ by increasing the drift component $\theta_j$, so $\pi_j(\cdot)$ represents an expected rate of \textit{cost increase} effected by increasing $\theta_j$ in the observed state, including both immediate and future costs. Temporarily suppressing the dependence of $\pi$ on $z$, we then restate the decomposed objective \eqref{eqn:decomposed-LP-objective} as
\begin{equation} 
\text{minimize } \ \sum_{j \in \mathcal{A}(k)} \pi_j \theta_j.
\label{eqn:FinalDecomp}
\end{equation}

\noindent \textbf{Interpreting the solution of the modified BCP in our original problem context.} To restate the decomposed optimization problem \eqref{eqn:FinalDecomp} in terms that are meaningful for the network model of original interest, let 
\begin{equation}
\mathcal{X}(k)=\left\{x_j \geq 0, \, j \in \mathcal{A}(k): \sum_{j \in \mathcal{A}(k)} x_j \leq 1\right\} \quad \text{for  } k=1,\ldots,p.
\end{equation}
Elements of $\mathcal{X}(k)$ represent feasible allocations of server $k$ capacity to the activities conducted by that server, ignoring the requirement that positive allocations can only be made to non-empty buffers. Thus, elements of $\mathcal{X}(k)$ have roughly the same role in our original network control problem as do elements of $\Theta(k)$ in the modified BCP. Rewriting the decomposed optimization problem \eqref{eqn:FinalDecomp} in terms of capacity allocations $x$, we argue in the paragraph to follow that the appropriate objective in the rewritten problem is to

\begin{equation} 
\text{maximize } \ \sum_{j \in \mathcal{A}(k)} \pi_j x_j.
\label{eqn:AlternateDecomp}
\end{equation}

The switch from minimization in \eqref{eqn:FinalDecomp} to maximization in \eqref{eqn:AlternateDecomp} results from our original definition of a deviation control as a \textit{decrement} from a nominal capacity allocation. Drift vectors in the modified BCP correspond to scaled deviation controls, so increasing a component of the BCP drift vector \textit{decreases} the corresponding capacity allocation. Considering that reversal in orientation, one can restate the interpretation of $\pi_j$ as follows: it is the rate at which one \textit{decreases} cost by increasing $x_j$ in the observed state, including both future and immediate costs. From this the objective \eqref{eqn:AlternateDecomp} is obvious.
\newline
\newline
\noindent \textbf{Proposed policy for the original discrete-flow network.} We use the method of \cite{ata2023drift} to solve the modified BCP, saving the gradient function $\nabla V(\cdot)$ for future use; to be precise, it is a neural network representation of $\nabla V(\cdot)$ that is saved.

To specify a control policy for the original queueing network model, we fix a time $t \geq 0$ and suppose that $Q(t)=q \in \mathbb{Z}_+^m$. That is, we denote by $q$ the $m$-vector of buffer contents (or queue lengths) observed at time $t$. The corresponding scaled state vector $z$ for the modified BCP is $z=r^{-1/2}q$, and we compute the indices $\pi_1(z),\ldots,\pi_n(z)$ given the current state via \eqref{eq:DefinePi}. Given those computed values, our task is to determine an $n$-vector $x$ of capacity allocations for use in state $q$. Based on the analysis above, we decompose that task into specifying a vector of capacity allocations for each server separately.

That requires consideration of an issue previously ignored, as follows. For each server $k=1,\ldots,p$, let $\mathcal{A}_k^*(q)$ be the set of all processing activities $j \in \mathcal{A}(k)$ such that $q_{b(j)}>0$. In words, $\mathcal{A}_k^*(q)$ is the set of all activities available to server $k$ whose associated buffer is non-empty. Our proposal is that capacity allocations be determined for each server $k$ separately by solving the maximization problem \eqref{eqn:AlternateDecomp}, but with positive allocations allowed only to activities $j \in \mathcal{A}_k^*(q)$. The solution of that very simple linear program is the following. 
\newline
\newline
\fbox{\parbox{1.0\textwidth}{In state $q$, each server $k=1,\ldots,p$ idles if either $\mathcal{A}_k^*(q)$ is empty or $\pi_j(z)<0$ for all $j \in \mathcal{A}_k^*(q)$. Otherwise, server $k$ devotes its full capacity to an activity $j \in \mathcal{A}_k^*(q)$ for which $\pi_j(z)$ is maximal.}}
\newline\newline
\newline
\noindent \textbf{Criss-cross example.} We now illustrate the proposed policy for the three examples introduced earlier, restricting attention throughout to the expected case where 
\begin{equation}
    \frac{\partial V}{\partial z_i} \geq 0 \ \ \text{for all } \ i.
    \label{eqn:positive-gradient}
\end{equation}
For the criss-cross network, recall that $\mathcal{A}(1) = \{1,2 \}$, $\mathcal{A}(2) = \{ 3\}$, and $\tilde c = 0$. Thus, using the matrix $R$ in \eqref{eq:CrissCrossMatrices}, formula \eqref{eq:DefinePi} gives
\begin{eqnarray*}
    \pi_1 = 2 \frac{\partial V}{\partial z_1}, \ \ \pi_2 = 2 \left( \frac{\partial V}{\partial z_2} - \frac{\partial V}{\partial z_3}\right) \ \textrm{ and } \ \pi_3 = \frac{\partial V}{\partial z_3}.
\end{eqnarray*}
Because there is a one-to-one correspondence between buffers and activities in this example, we will refer to buffers and activities interchangeably. Given that (\ref{eqn:positive-gradient}) holds, our proposed policy dictates the following. First, server 2 works on buffer 3 except when it is empty, as one would expect. Second, if buffers 1 and 2 are both non-empty, then server 1 works on whichever of them has the larger index $\pi_j$. If buffer 2 is empty, then server 1 works on buffer 1, given that (\ref{eqn:positive-gradient}) holds. Finally, if buffer 1 is empty but buffer 2 is not, it is still possible that server 1 will choose idleness, because $\pi_2$ could be negative even when (\ref{eqn:positive-gradient}) holds.

Each of these prescriptions is more or less obvious from interpreting $\partial V / \partial z_i$ as the rate of decrease in expected total cost per unit of decrease in buffer $i$'s content. (This is the ``shadow price" interpretation of the gradient function $\partial V / \partial z$.) Thus, for example, an increment of capacity that server 1 devotes to activity 2 has the effect of decreasing $q_2$ but increasing $q_3$, each at expected rate 2 (the average service rate for that activity), and the formula for $\pi_2$ converts that physical description into a rate of change in expected total cost.
\newline
\newline
\noindent \textbf{Pesic-Williams Example.} Here the activities available to the three servers are  $\mathcal{A}(1) = \{1,5 \}$, $\mathcal{A}(2) = \{  2,  3\}$ and $\mathcal{A}(3) = \{ 4 \}$. Also, there are no penalties associated with server idleness in this example ($\tilde c = 0)$. Combining that data with the average service rates and routing information in Figure \ref{fig:pesic}, formula \eqref{eq:DefinePi} gives the following:
\begin{eqnarray*}
    \pi_1 = \frac{\partial V}{\partial z_1}, \ \pi_2 = 2 \frac{\partial V}{\partial z_1}, \ \pi_3 = 2 \frac{\partial V}{\partial z_2}, \ \pi_4 = \frac{\partial V}{\partial z_3} \ \text{ and } \ \pi_5 = \frac{\partial V}{\partial z_3}.
\end{eqnarray*}
Once again, each server $k=1,2,3$ checks activities $j \in \mathcal{A}(k)$, eliminates from consideration those whose corresponding buffers are empty, and devotes all of its capacity to whichever of the remaining ones has the largest index $\pi_j$. Given that (\ref{eqn:positive-gradient}) holds, we have $\pi_j \geq 0$ for all $j$, so a server will idle only when all of its constituent buffers are empty. 
\newline
\newline
\noindent \textbf{Three-Station Example.} In this example, servers 1,2,3 are processing servers, whereas those numbered 4,5,6 are fictional ``input servers."  Also, the activities available to the various servers are
\begin{equation*}
\mathcal{A}(1) = \{ 1, 7 \}, \quad \mathcal{A}(2) = \{ 2,5 \}, \quad \mathcal{A}(3) = \{ 3, 4, 6, 8 \}, \quad \mathcal{A}(4) = \{ 9, 10 \}, \quad \mathcal{A}(5) = \{ 11 \}, \quad \mathcal{A}(6) = \{ 12 \}.
\end{equation*}

For the processing activities $j=1,\ldots,8$, index values $\pi_j$ are calculated and used exactly as in our preceding examples. For the input activities, formula \eqref{eq:DefinePi} gives
\begin{equation*}
\pi_9 = \tilde c_4 - \frac{1}{2} \frac{\partial V}{\partial z_1}, \quad
\pi_{10} = \tilde c_4 - \frac{1}{2} \frac{\partial V}{\partial z_3}, \quad
\pi_{11}= \tilde c_5 -  \frac{1}{4}\frac{\partial V}{\partial z_4}, \quad \text{and} \quad
\pi_{12}= \tilde c_6 - \frac{1}{4} \frac{\partial V}{\partial z_6}. 
\end{equation*}
For these activities, $\pi_j$ is best understood as the rate at which one \textit{increases} expected total cost by \textit{decreasing} capacity allocated to activity $j$. To see what that actually means, consider activity 9, conducted by the fictional server 4, which corresponds to accepting type A arrivals and directing them to buffer 1. To evaluate the expected cost of a decrease in that activity, one needs to use three different factors: first, the cost per time unit of idling server 4, which is the cost per time unit of ``turning off" the incoming stream of type A jobs, which is $\tilde c_4$; second, the average rate at which activity 9 generates input to buffer 1, which is the average arrival rate for type A jobs, which is $1/2$; and third, the rate at which increasing the content of buffer 1 increases expected total cost, which is $\partial V / \partial z_1$. Note that the cost of ``turning off" an input stream must be expressed in \text{scaled} terms (that is, we use $\tilde c_4$ rather than $c_4$ in the formula for $\pi_9$), because the partial derivatives of $V$ are expressed in scaled terms.

In any given state, server 4 calculates the values of $\pi_9$ and $\pi_{10}$. If at least one of them is $\geq 0$, then the server ``allocates its full capacity" to activity 9 or activity 10, depending on which has the larger index value. In the former case, this means that arriving jobs of type A are accepted rather than rejected, and they are sent to buffer 1, thereafter following the upper route in Figure \ref{fig:three-station}. In the latter case, type A arrivals are accepted and sent along the lower route, starting in buffer 3. If both $\pi_9$ and $\pi_{10}$ are negative in the observed state, then type A arrivals are rejected.

For the other two input streams, admission control is simpler: type B arrivals are accepted by server 5 if  $\pi_{11} \geq 0$ (that is, if $4 \, \tilde c_5 \geq \partial V / \partial z_4$), and are rejected otherwise; similarly, server 6 accepts type C arrivals if $4\,\tilde c_6 \geq \partial V/\partial z_6$, and rejects them otherwise.

\section{Computational Results}
\label{sec:Tests}

In this section, we solve the modified BCP for each of our three examples, using the problem formulation developed in Section~\ref{sec:drift1} and the computational method of \citet{ata2023drift}. The resulting policy is then implemented in a simulation of the original discrete-flow network, as described in Section~\ref{sec:policy}. In each case, we compare the cost performance of our proposed policy for the discrete-flow network against a best available benchmark. To be specific, the performance measure that we record for each policy in each problem setting, referred to as both the policy's  ``cost" and its ``objective value" at different points in the text below, is the average discounted sum of all costs incurred over an infinite horizon, with the system starting empty in each case.

For the first two examples, the original (exact) network control problem is a Markov decision process (MDP) under the assumptions made here, and we are able to solve that problem numerically. We then use the optimal MDP policy as our benchmark. For the three-station network example, which has an eight-dimensional state descriptor, numerical solution of the MDP is not feasible. We therefore use a greedy heuristic policy (a cost-aware variant of the maximum-pressure policy) as our benchmark; for added perspective, the performance of that greedy heuristic will also be reported for the first two examples.

\subsection{Criss-cross network}
\label{sec:criss-cross}
\begin{figure}
    \centering
    \includegraphics[width=0.7\linewidth]{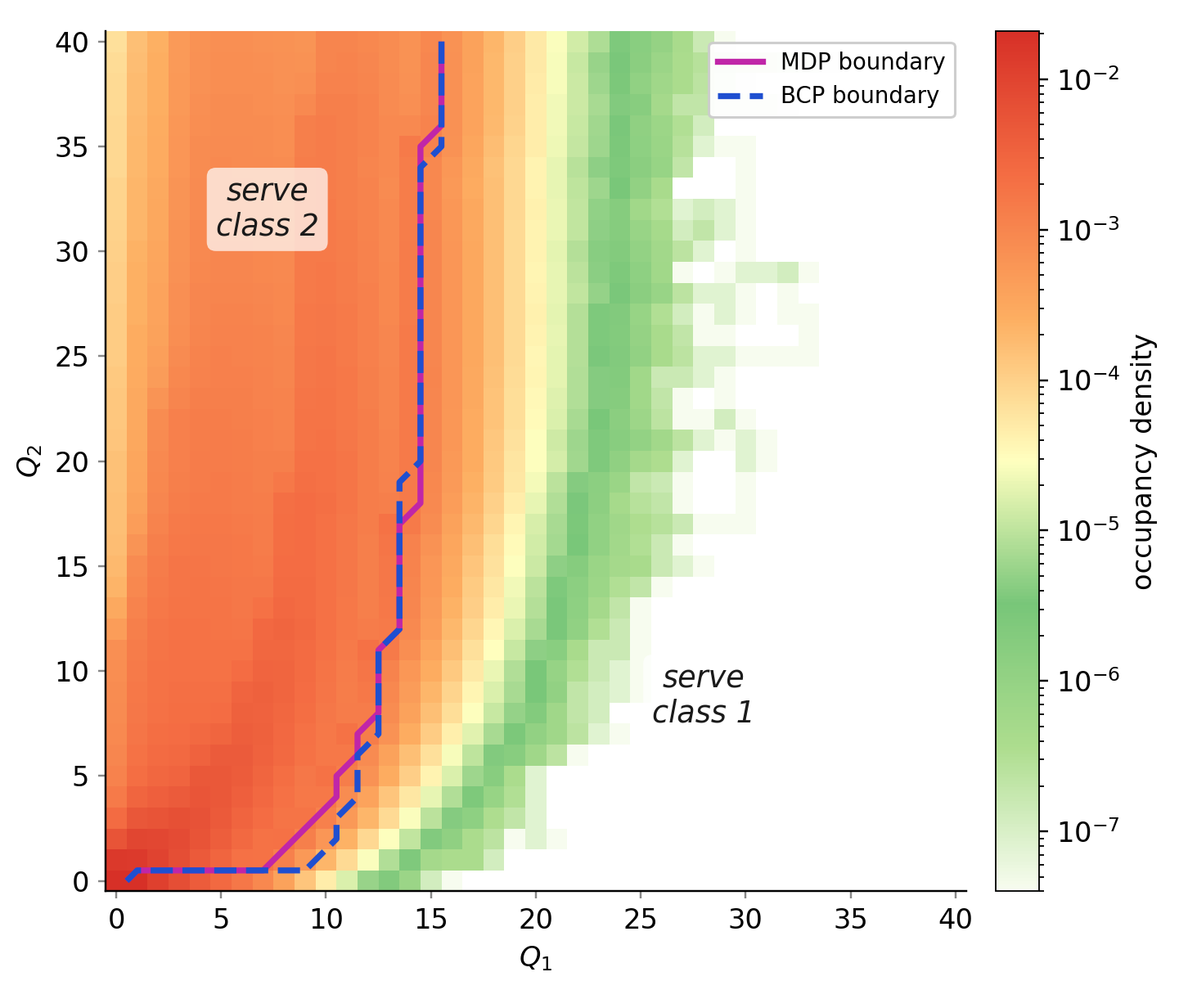}
    \caption{MDP (purple, solid) and BCP (blue, dashed) scheduling boundary on the slice $Q_3=0$. 
    Background shading is the MDP-policy state-occupancy density conditioned on $Q_3=0$ (logarithmic color scale).}
    \label{fig:Martins-BCP}
\end{figure} 
 
We first consider the criss-cross network example depicted in Figure~\ref{fig:criss-cross}, for which the BCP and EWF were described in Subsection \ref{sec:CrissCross}. To obtain a benchmark for the discrete-flow network, we numerically solve the associated MDP. Figure \ref{fig:Martins-BCP} compares the optimal MDP policy with our proposed policy for the discrete-flow network, which we refer to as the BCP policy. Motivated by \eqref{eqn-ssc-criss-cross}, one expects $Q_3(t) \approx 0$ whenever the workload vector falls below the dotted white line in Figure \ref{fig:crisscross-EWF}. Therefore, we focus attention on the case $ Q_3 = 0$ in Figure \ref{fig:Martins-BCP}  and display the optimal action as a function of $Q_1$ and $Q_2$. The solid purple and dashed blue curves represent the switching curves for the MDP and BCP policies, respectively: when the queue-length process lies above the switching curve, server 1 gives priority to class 2; and when it falls below the switching curve, server 1 gives priority to class 1, thereby starving server 2 and causing it to idle. The background shading in Figure \ref{fig:Martins-BCP} shows the relative amount of time that the system spends in each state under the MDP policy when starting empty, with darker colors indicating higher state frequencies.

The two policies pictured in Figure \ref{fig:Martins-BCP} are very similar, and Figure 6b of \cite{ata2024singular} provides a similar comparison of the MDP policy with yet another policy for the criss-cross example, one that was derived by solving the EWF and interpreting that solution in the manner proposed by \cite{martins1996heavy}. We refer to the latter as \textit{the singular control policy}. 

Table \ref{tab:criss-cross} reports the simulated cost performance, with standard errors, of four policies for the criss-cross queueing network: the greedy heuristic policy that is described and motivated in Appendix \ref{app:computation}; the optimal MDP policy; the singular control policy with safety stock parameter $s^*=1$; and the BCP policy proposed in Section \ref{sec:policy}. The singular control and BCP policies each achieve an objective value that differs from that of the optimal MDP policy by just a fraction of one percent. The objective value for the greedy heuristic is higher by approximately $6\%$.
  
\begin{table}[!ht] \caption{Simulated cost performance of four policies for the criss-cross network.} \label{tab:criss-cross} \centering \begin{tabular}{@{}lr@{}} \toprule \textbf{policy} & \textbf{cost} \\ \midrule MDP & $1681.7 \pm 1.5$ \\ greedy heuristic & $1789.4 \pm 2.3$ \\ singular control $(s^*=1)$ & $1690.3 \pm 1.5$ \\ BCP & $1686.6 \pm 2.1$ \\ \bottomrule \end{tabular} \end{table}

\subsection{Pesic-Williams parallel-server example}
\label{sec:parallel-server}
We turn next to the Pesic-Williams parallel-server example displayed in Figure~\ref{fig:pesic}. As in the criss-cross network example, we numerically solve the associated MDP to obtain a benchmark policy. Figure \ref{fig:PW-BCP} compares this benchmark policy with our proposed policy, given the cost vector $\tilde{h} = (1,2,3)$. In this case, an optimal solution of the EWF holding-cost minimization problem \eqref{eq:PWcost} is $z_1 = w_1, \, z_2 =0, \, z_3 = w_2$. That is, it is best to keep all of the workload $w_1$ in buffer 1, setting the buffer 2 queue length to zero.

Motivated by this, we focus on the case $Q_2 = 0$ in Figure \ref{fig:PW-BCP}, showing how server 1 switches between the basic activity 1 (serving buffer 1) and the nonbasic activity 5 (serving buffer 3) as a function of $Q_1$ and $Q_3$. This behavior is described by a switching curve: the solid and dashed blue curves represent the switching curves for the MDP and BCP policies, respectively; when the queue-length process lies above the switching curve, server 1 serves buffer 3, and otherwise, it serves buffer 1. The figure's background shading shows the relative amount of time the system spends in each state when initially empty and following the MDP policy, with darker colors indicating higher state frequencies.

\begin{figure}
    \centering
    \includegraphics[width=0.7\linewidth]{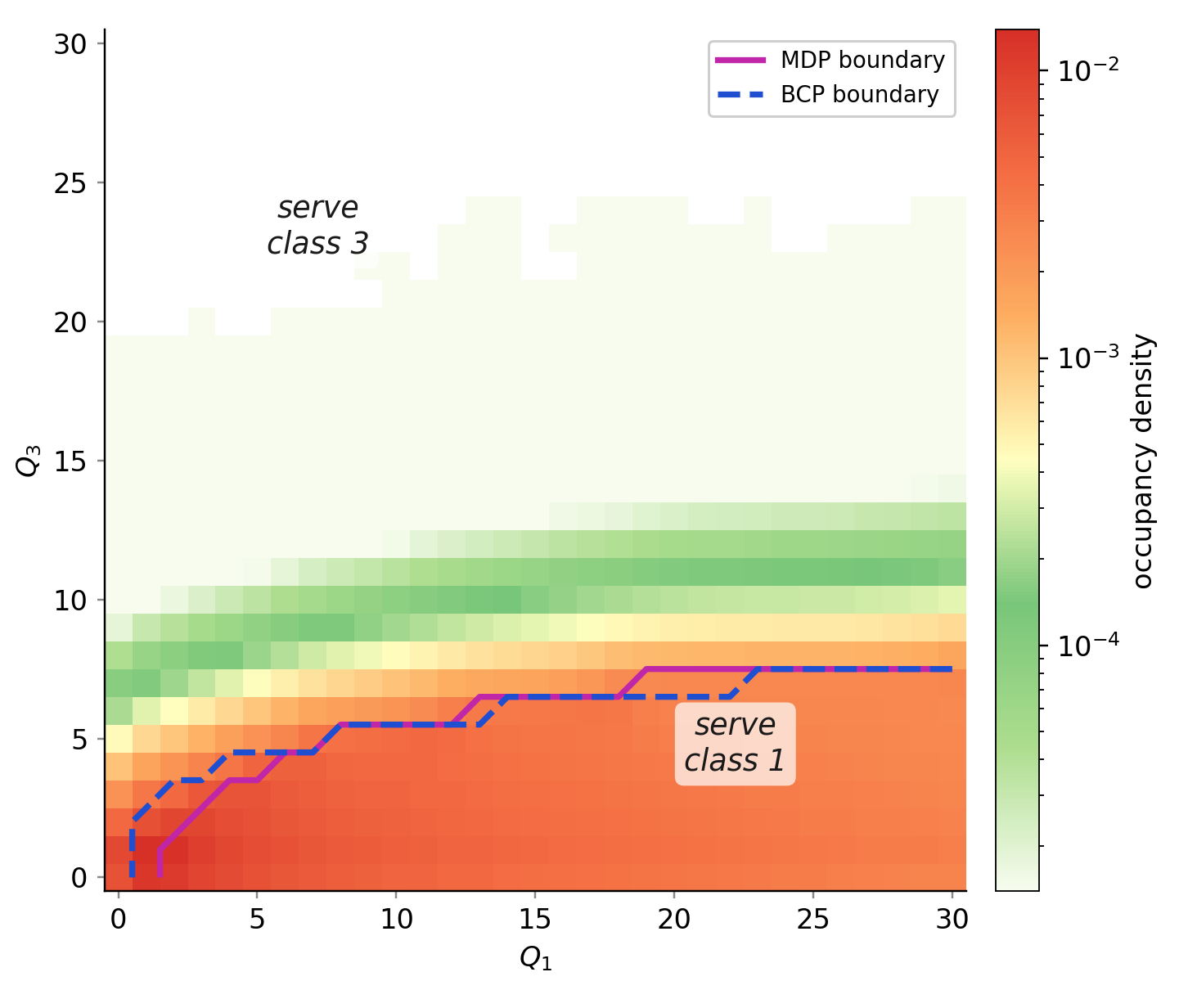}
    \caption{MDP (purple, solid) and BCP (blue, dashed) scheduling boundaries on the slice $Q_2=0$, given the holding cost vector $\tilde{h} = (1,2,3)$. 
    Background shading is the MDP-policy state-occupancy density conditioned on $Q_2=0$ (logarithmic color scale).}
    \label{fig:PW-BCP}
\end{figure}
 
Table \ref{tab:Pesic-Williams} reports the simulated cost performance, with standard errors, of three policies for the Pesic-Williams discrete-flow model, again assuming the holding cost vector $\tilde{h} = (1,2,3)$: they are the greedy heuristic policy (see Appendix \ref{app:computation}), the optimal MDP policy, and the BCP policy proposed in Section \ref{sec:policy}. The BCP policy's objective value is very close to that of the optimal MDP policy, within a fraction of one percent, and it is significantly below that of the greedy heuristic.

\begin{table}[!ht] \caption{Simulated cost performance of three policies for the Pesic-Williams parallel-server example.} \label{tab:Pesic-Williams} \centering \begin{tabular}{@{}lr@{}} \toprule \textbf{policy} & \textbf{cost} \\ \midrule MDP & $2581.2 \pm 2.2$ \\ greedy heuristic & $3277.1 \pm 3.9$ \\ BCP & $2586.2 \pm 3.0$ \\ \bottomrule \end{tabular} \end{table}

\subsection{Three-station example}
\label{sec:three-station}

Lastly, we consider the 3-station example introduced in Section \ref{sec:InputControl}. Because solving the MDP is computationally infeasible for this example, Table \ref{tab:three-station} compares the BCP policy proposed in Section \ref{sec:policy} with the greedy heuristic policy that is described and motivated in Appendix \ref{app:computation}. The
greedy policy’s cost is approximately $9\%$ higher than that of the BCP policy.

\begin{table}[!ht] \caption{Simulated cost performance of two policies for the three-station example.} \label{tab:three-station} \centering \begin{tabular}{@{}lr@{}} \toprule \textbf{policy} & \textbf{cost} \\ \midrule greedy heuristic & $8844.8 \pm 10.4$ \\ BCP & $8117.4 \pm 10.1$ \\ \bottomrule \end{tabular} \end{table}

\section{Conjectures regarding two kinds of convergence}
\label{sec:Discussion}

\noindent First, let $V(z)$ denote the minimum cost achievable in our initial BCP (Section \ref{sec:raw}), given that $Z(0)=z$, calling $V(\cdot)$ the \textit{optimal value function} for that problem. Also, let $V_a(\cdot)$ be defined similarly, but for the modified BCP (Section \ref{sec:drift1}), assuming for ease of exposition that $a_1=\ldots=a_n=a>0$. Because every admissible control for the modified BCP is also admissible for the original BCP, it is obvious that $V_a(\cdot) \geq V(\cdot)$. Now we argue that, furthermore,
\begin{equation}
   V_a(\cdot) \downarrow V(\cdot) \text{\quad as \quad} a \uparrow \infty.
\label{eq:converge}
\end{equation}

For an informal justification of this assertion, let
\begin{equation}
   \Delta := \{\delta \in \mathbb{R}^m : \delta = Ry, y \in \mathbb{R}^n, Ky \geq 0\}.
\label{eq:feasible_directions}
\end{equation}
The $m$-vectors $\delta \in \Delta$ are the \textit{feasible directions of control} in our initial BCP: starting from an arbitrary state $z \in \mathbb{R}^m_+$ at any given time, the system manager in the initial BCP can enforce an instantaneous displacement to another state $z+\delta \in \mathbb{R}^m_+$ if and only if $\delta \in \Delta$. The system manager in the \textit{modified} BCP can similarly enforce displacement in any such direction, but only at a \textit{finite rate}. As the bound on that rate is increased toward infinity, the argument goes, the distinction between the two problems becomes insignificant.

What that argument may seem to ignore is the following: our formulation of the modified BCP demands that the system manager effect displacements in direction $RQ^i$ whenever the boundary $Z_i=0$ is struck ($i=1,\ldots,m$), and the matrix $Q$ has not been ``optimized" in any sense. But if the upper bound parameter $a$ is large, then by exercising displacement control at the maximum allowable rate when $Z$ is \textit{near} the boundary, the system manager can produce behavior that closely approximates reflection at the boundary in any desired direction. That is, as $a \uparrow \infty$, the system manager can effectively ``overpower" the pre-specified policy for action at the boundary. Of course, the same reasoning applies if the upper bound parameters $a_1,\ldots,a_n$ are different but are scaled up by a common factor that $\uparrow \infty$.
\ 

Our second conjecture concerns a prelimit analog of the modified BCP. Recall from (\ref{eq:more_bounds}) that, to faithfully represent our original problem context, the drift-rate bounds $a_1,\ldots,a_n$ should be of order $\sqrt{r}$ in the $r^{th}$ system. With that scaling in mind, consider the $r^{th}$ network control problem with controls restricted to the scaled analog of (\ref{eq:formofY}), that is,
\begin{equation*}
d \tilde{Y}(t) = \tilde \theta(t) \, dt + Q \, d\tilde{L}(t),
\end{equation*}
where the drift-rate process $\tilde \theta(\cdot)$ is bounded by parameters of order $\sqrt{r}$, and $\tilde{L}_i(\cdot)$ increases only when buffer $i$ is empty ($i=1,\ldots,m$). Substituting this expression for $\tilde{Y}(\cdot)$ into the system equation (\ref{eq:centered_scaled_system}), while retaining the constraints (\ref{eq:Monotonicity}) and the objective (\ref{eq:scaledobjective}), one obtains what may be viewed as a prelimit formulation of the modified BCP. Letting $V^r(0)$ denote the optimal objective value for that formulation when the system is initially empty, we conjecture that $V^r(0) \rightarrow V(0)$ as $r \rightarrow \infty$. Taken together, the two conjectures support our computational approach from complementary directions: the first one says that the modified BCP recovers the initial BCP as the drift-rate bounds grow large, and the second one says that, with those bounds scaled in proportion to $\sqrt{r}$, the modified BCP correctly captures the behavior of the network control problem of original interest in the heavy traffic regime.

\bibliographystyle{plainnat}
\bibliography{mybib}
\begin{appendix}

\section{Nominal activity rates and the heavy traffic parameter regime}
\label{Nominal}
The notion of ``heavy traffic" has been developed for dynamic control problems, as opposed to purely descriptive queueing models, in \citet{harrison2000brownian}  and \citet{harrison2003broader}.  In this appendix, we briefly paraphrase that development in terms that are appropriate for analysis of a single given system; in contrast, most of the literature on heavy traffic approximations has been focused on formal limit theory to justify the use of diffusion-based models in an asymptotic parameter regime.

As in Section 2, let $\lambda$ be the $m$-vector of exogenous arrival rates (or exogenous demand rates), $R$ the $m \times n$ input-output matrix, and $A$ the $p \times n$ capacity consumption matrix for a network control problem. Now consider the following \textit{static planning problem}:
\begin{equation}
\text{Choose  } \rho \text{ and } x \geq 0  \text{  to minimize } \rho \text{  subject to  }  Rx=\lambda, \, \text{ and } Ax \leq \rho e. 
\label{eq:SPP}
\end{equation}
Here the decision variables $\rho$ and $x$ are scalar and $n$-dimensional, respectively, and $e$ is the $p$-vector of ones. For an interpretation of this linear programming problem, one considers a deterministic fluid analog of our original network model, with external inputs, activity-generated outputs, and capacity consumption all occurring continuously at rates dictated by the original model data. One interprets $x$ as the constant vector of activity levels employed by a system manager, and $\rho$ as the maximum level of capacity utilization implied by that choice of $x$.

Assuming for simplicity that \eqref{eq:SPP} has a unique optimal solution, we denote by $\rho^*(\lambda)$ and $x^*(\lambda)$ those optimal values, thus emphasizing the pivotal role of the exogenous input rates. The standard or default interpretation of the phrase ``heavy traffic parameter regime" is embodied in the assumption immediately below; another, more general interpretation of that phrase will be needed later in this appendix for treatment of our three-station example.
\newline
\newline
\noindent \textbf{Standard Heavy Traffic Assumption.} There exists an $m$-vector $\lambda^* \geq 0$ that is close to $\lambda$ (in the usual Euclidean metric) and such that $\rho^*(\lambda^*)=1$ and $Ax^*(\lambda^*)=e$.
\newline
\newline
We then take the vector $\beta$ of nominal activity rates (see Section \ref{sec:Centering}) to be
\begin{equation}
\beta=x^*(\lambda^*).
\end{equation}
\label{eq:ChooseBeta}
\indent In words, our Standard Heavy Traffic Assumption says the following: if the vector of exogenous input rates is changed slightly from $\lambda$ to $\lambda^*$, then every server will have to work at full capacity (that is, $Ax^*(\lambda^*)=e$) to avoid inventory build-ups over time. This situation, where a small adjustment in external arrival rates causes \textit{all} servers to be “critically loaded,” may seem overly restrictive, but the assumption really means that lightly loaded servers, if any such exist, have simply been deleted from our model. That is, the model under discussion
provides explicit representation of just the “bottleneck sub-network” composed of $p$ heavily loaded servers.
\newline
\newline
\noindent \textbf{Criss-Cross example.} In this simple network model, the system manager has no discretion as to which server processes work from which buffer, and there is no input control capability. Thus, combining the first-order data \eqref{eq:criss-crossData} with the routing information conveyed by Figure \ref{fig:criss-cross}, we calculate a \textit{load factor} of $\lambda_1/\mu_1+\lambda_2/\mu_2=0.975$ for server 1 and $\lambda_2/\mu_3=0.95$ for server 2. One sees that a small increase in $\lambda_2$ from $0.95$ to $1$ will push both load factors up to $1$, so we are led to choose
\begin{equation}
\lambda^* = (1,1,0) \quad \text{and} \quad \beta=x^*(\lambda^*)=(0.5, 0.5,1).
\label{eq:criss-crossLoads}
\end{equation}

\noindent \textbf{The Pesic-Williams parallel-server example.} Given the routing structure pictured in Figure \ref{fig:pesic} and the parameter values specified in \eqref{eqn-pesic-williams-arrival-service-rates}, one sees that our Standard Heavy Traffic Assumption can be satisfied by choosing the nearby vector of exogenous arrival rates $\lambda^*=(2,1,1)$, which dictates the following vector of nominal activity rates:
\begin{equation}
\beta=x^*(\lambda^*)=(1, 0.5,0.5,1,0).
\label{eq:PWLoads}
\end{equation}
That is, the nominal processing plan calls for server 1 to dedicate all its capacity to buffer 1, server 3 to devote all its capacity to buffer 3, and server 2 to share its capacity equally between buffers 1 and 2. A notable feature of this plan is that activity 5 (server 1 processing jobs from buffer 3) is not used; it is a nonbasic activity. 
\newline
\newline
\noindent \textbf{Three-station example.} In this more complex model (see Subsection \ref{sec:InputControl}), the system manager can ``turn off" any of the external arrival processes, but doing so is costly. To model this situation, we have set all three exogenous arrival rates to zero ($\lambda=0$) as a matter of convention, and created four ``input activities" that are conducted by three fictional, capacity-constrained ``input servers." 

With that model structure, the penalties associated with rejecting arrivals are represented as ``idleness costs" for our fictional input servers. Of course, the system manager's objective is \textit{not} simply one of handling an exogenously imposed load subject to capacity constraints, but rather one of minimizing costs associated with available activities. \citet{harrison2003broader} has developed a general framework for analysis of such problems, but in the interest of brevity, the associated general theory will not be recapitulated here. Rather, what follows is a minimal treatment of the three-station example that is our immediate concern.

For the class of models considered in this paper, the general framework referred to above involves the following alternative to the standard static planning problem \eqref{eq:SPP}: 

\begin{equation}
\text{Choose  } x \geq 0  \text{  to minimize } c \cdot u \text{  subject to  }  Rx=\lambda, \, u=e-Ax \text{ and } u \geq 0. 
\label{eq:SPP2}
\end{equation}
Here $c$ is the $p$-vector of cost rates associated with server idleness (including idleness of the fictional input servers), and $e$ is the $p$-vector of ones, while $R$ and $A$ have the dimensions and interpretation cited earlier in this appendix; see Equations (\ref{eqn:Rvalues-3-station}) and (\ref{eqn:Avalues-3-station}) for the $R$ and $A$ matrices, respectively, for the three station example. For our three-station example, where $\lambda=0$ as a modeling convention, the constraint $Rx=\lambda$ means that the eight true processing activities (that is, the processing of units in the eight buffers by the associated servers) are conducted at levels that precisely balance the rates chosen for the four input activities. Components of $u$ represent server idleness rates, and only the last three components of the six-vector $c$ (that is, the idleness cost rates for the three input servers) are positive. For our three-station example, there is a unique optimal solution $x^*$ for the static planning problem \eqref{eq:SPP2}, which we adopt as our nominal processing plan $\beta$. To be specific, that solution is
\begin{equation}
\beta=x^*=(0.5,0.75,0.25,0.25,0.25,0.25,0.5,0.25,0.5,0.5,1,1).
\label{eq:three-station-nominal-plan}
\end{equation}

The last four components of $\beta$ specify that, under the nominal processing plan, type A input is divided equally between buffers 1 and 3, with all of that input accepted, and that all input of types B and C is accepted as well; the first eight components specify that each of the three original servers (as opposed to fictional input servers) uses all of its available capacity, dividing that capacity among its client buffers in proportions that enable a flow rate of 0.25 through each buffer. Thus, in the idealized fluid model underlying the static planning problem \eqref{eq:SPP2}, no input control costs would ever be incurred, nor would there be any inventory holding costs, since all buffer contents would remain precisely at zero.

\section{Additional calculations under standard assumptions}
\label{Additional}

The additional ``standard" assumptions listed in Section \ref{sec:general}  are threefold: each exogenous input process is a renewal process; service times for each activity form an i.i.d. sequence; and jobs change class in Markovian fashion after completing service. Under these assumptions, one has the following rather obvious formula for the $m \times n$ matrix $R$ of first-order data in (\ref{eq:first_order}): 

\begin{equation}
   R_{ij} =
\begin{cases} 
\text{      }\mu_j & \text{if  } i = b(j), \\
-\mu_j  P_{ij} & \text{otherwise},
\end{cases}
\label{eq:redefineR}
\end{equation}
where $\mu_j$ is the mean service rate (that is, the reciprocal of the mean service time) for activity $j$, and $P_{ij}$ is the probability that a job of class $b(j)$, after it completes a type $j$ service, will next become a class $i$ job.

Next, because the exogenous input processes are independent renewal processes, a standard result in renewal theory gives the following formula for the covariance matrix $\Gamma^0$ in \eqref{eq:second_order}:
\begin{equation}
 \Gamma^0=\text{diag }( (\lambda_1^*)^3\eta_1^2,\ldots,(\lambda_m^*)^3\eta_m^2),
\label{eq:Gamma0}
\end{equation}
where $\lambda_i^*$ is the average exogenous arrival rate into class $i$, and $\eta_i^2$ is the variance of inter-arrival times for that class. For the important special case where the exogenous input processes are independent \textit{Poisson} processes, one has $\eta_i=1/\lambda_i^*$ for each $i=1,\ldots,m$, so \eqref{eq:Gamma0} specializes to 
\begin{equation}
\Gamma^0=\text{diag}(\lambda_1^*,\ldots,\lambda_m^*) \quad \text{for the Poisson case}.
\label{eq:PoissonCase}
\end{equation}

Finally, by generalizing in an obvious way the calculation done in Section 4 of \cite{harrison1988brownian}, one obtains the following formula for the covariance matrices $\Gamma^1,\ldots,\Gamma^m$ in \eqref{eq:DefineX}:
\begin{equation}
 \Gamma^j=\mu_j\Omega^j+\mu_j\sigma_j^2R^j(R^j)' \quad \text{for} \, j=1,\ldots,m.
\label{eq:Covariance}
\end{equation}
Here $R^j$ denotes the $j^{th}$ column of $R$ as usual, and $\Omega^j$ is an $m \times m$ covariance matrix with the following elements (using the notation $\delta_{ij}=1$ if $i=j$ and $=0$ otherwise):
\begin{equation}
\Omega^j_{ik}=P_{ij}(\delta_{ik}-P_{kj}) \quad \text{for} \quad i,k=1,\ldots,m \quad \text{and} \quad j=1,\ldots,n.
\label{DefineOmega}
\end{equation}
Readers wishing to check formula \eqref{DefineOmega} must pay careful attention to differences in model formulation and notation used in this paper versus \cite{harrison1988brownian}: our model here is more general, with $n \geq m$ rather than $n=m$, and the quantity denoted $P_{ij}$ in \cite{harrison1988brownian} is here denoted $P_{ji}$.

\section{Verifying that the matrix $H$ is completely-S}
\label{Verifying}

In this appendix we verify, for each of our three examples, that the matrix $H=RQ$ specified in Section \ref{sec:drift2} is completely-$S$ for every $\kappa\in(0,1)$. In each case we prove the stronger statement that $H$ is a \textit{$P$-matrix}, meaning that all of its principal minors are strictly positive. The desired conclusion then follows immediately: every principal submatrix of a $P$-matrix is again a $P$-matrix, and every $P$-matrix is an $S$-matrix, so every $P$-matrix is completely-$S$. Throughout this appendix, $H[\alpha]$ denotes the principal submatrix of $H$ obtained by retaining the rows and columns indexed by a set $\alpha$, while deleting others.
\newline
\newline
\noindent \textbf{Criss-cross example.} The matrix $H$ for this example, displayed in \eqref{eq:defineH-matrix}, is $3\times 3$, so there are $2^3-1=7$ principal minors to check. The order-1 principal minors are the diagonal elements of $H$, each of which equals 1. For the order-2 principal minors, direct calculation gives
\[
\det H[\{1,2\}] = 1-\kappa^2, \qquad
\det H[\{1,3\}] = 1, \qquad
\det H[\{2,3\}] = 1.
\]
Finally, because the only nonzero element in the third column of $H$ is $H_{33}=1$, expanding the determinant along that column gives
\[
\det H = \det H[\{1,2\}] = 1-\kappa^2.
\]
Thus all seven principal minors are strictly positive for $\kappa\in(0,1)$, so $H$ is a $P$-matrix and hence is completely-$S$. (If one takes $\kappa=1$, then the first two rows of $H$ are negatives of one another, so there exists no $x\geq 0$ satisfying $Hx>0$; that is, $H$ is not an $S$-matrix in that case.)
\newline
\newline
\noindent \textbf{Pesic-Williams parallel-server example.} The matrix $H$ for this example, displayed in \eqref{eq:HforPW}, is again $3\times 3$. Its order-1 principal minors are the diagonal elements $H_{11}=2$, $H_{22}=1$ and $H_{33}=1$, and its order-2 principal minors are
\[
\det H[\{1,2\}] = 2-\kappa^2, \qquad
\det H[\{1,3\}] = 2, \qquad
\det H[\{2,3\}] = 1.
\]
As in the previous example, the only nonzero element in the third column of $H$ is $H_{33}=1$, so
\[
\det H = \det H[\{1,2\}] = 2-\kappa^2.
\]
Again, all seven principal minors are strictly positive for $\kappa\in(0,1)$, so $H$ is a $P$-matrix and hence is completely-$S$.
\newline
\newline
\noindent \textbf{Three-station example.} The matrix $H$ for this example, displayed in \eqref{eq:ThirdH}, is $8\times 8$, so there are $2^8-1=255$ principal minors to check; we have verified all of them using exact rational arithmetic. Every principal minor of $H$ factors into a positive rational constant multiplied by a product of powers of the five linear factors
\[
(1-\kappa),\qquad
(1+\kappa),\qquad
(3-\kappa),\qquad
(3+\kappa),\qquad
(3-2\kappa).
\]
Since all five factors are strictly positive for $\kappa\in(0,1)$,
every principal minor is strictly positive. Therefore, $H$ is a
$P$-matrix and hence is completely-$S$.

For illustration, several representative principal minors are given below:
\[
\begin{array}{ccl}
\hline
\text{Order} & \text{Index set } \alpha & \det H[\alpha]\\
\hline
1
& \{i\}\text{ for any }i
& \dfrac14
\\[1.5ex]

2
& \{1,7\}
& \dfrac{(1-\kappa)(1+\kappa)}{16}
\\[1.5ex]

2
& \{2,5\}
& \dfrac{(1-\kappa)(1+\kappa)}{16}
\\[1.5ex]

2
& \{3,4\}
& \dfrac{(3-\kappa)(3+\kappa)}{144}
\\[1.5ex]

3
& \{3,4,6\}
& \dfrac{(3+\kappa)^2(3-2\kappa)}{1728}
\\[1.5ex]

4
& \{3,4,6,8\}
& \dfrac{(1-\kappa)(3+\kappa)^3}{6912}
\\[1.5ex]

4
& \{1,2,5,7\}
& \dfrac{(1-\kappa)^2(1+\kappa)^2}{256}
\\[1.5ex]

6
& \{1,2,3,4,5,7\}
& \dfrac{(3-\kappa)(1-\kappa)^2
              (1+\kappa)^2(3+\kappa)}
             {36864}
\\[1.5ex]

7
& \{1,2,3,4,5,6,7\}
& \dfrac{(1-\kappa)^2(1+\kappa)^2
              (3+\kappa)^2(3-2\kappa)}
             {442368}
\\[1.5ex]

8
& \{1,2,3,4,5,6,7,8\}
& \dfrac{(1-\kappa)^3(1+\kappa)^2
              (3+\kappa)^3}
             {1769472}
\\
\hline
\end{array}
\]

To summarize, the matrix $H=RQ$ is a $P$-matrix for each of our three examples and for every $\kappa\in(0,1)$, and therefore $H$ is completely-$S$ in all three cases, as claimed in Section \ref{sec:drift2}.
\ 
\newline

\noindent \textbf{An alternative argument.} We conclude this appendix with a second, unified proof that H is completely-S in all three examples. Recall that A is a \textit{nonsingular M-matrix} if $A=sI-N$, where $N\geq 0$ elementwise and $s>\rho(N)$, with $\rho(\cdot)$ denoting the spectral radius. Such a matrix has a nonnegative inverse, and each of its principal submatrices is itself a nonsingular M-matrix. Below, $\mathbf{1}$ denotes the vector of all ones.
\begin{lemma}
\label{lem:MMproduct}
If $A$ and $B$ are nonsingular $M$-matrices of the same size, then $AB$ is completely-$S$.
\end{lemma}

\begin{proof}
Fix a nonempty index set $\alpha$. By Lemma~3.3 of \citet{grundy2007products}, $(AB)[\alpha] \geq A[\alpha] B[\alpha]$ elementwise. Moreover,
\[
x:=\left(A[\alpha]B[\alpha]\right)^{-1}\mathbf{1}
  =B[\alpha]^{-1}A[\alpha]^{-1}\mathbf{1}>0,
\]
because $A[\alpha]$ and $B[\alpha]$ are nonsingular $M$-matrices. Therefore $(AB)[\alpha] x \geq A[\alpha] B[\alpha] x = \mathbf{1}>0$, so every principal submatrix of $AB$ is an $S$-matrix.
\end{proof}

In each of our three examples, the elements of $H$ depend affinely on $\kappa$, and direct computation verifies the factorization
\begin{equation}
H = H^0 \left( I-\kappa \Phi \right),
\label{eq:MMfactorization}
\end{equation}
where $H^0$ denotes the matrix obtained from $H$ by setting $\kappa=0$, and $\Phi:=\kappa^{-1}(H^0)^{-1}(H^0-H)$ does not involve $\kappa$. For the criss-cross example \eqref{eq:defineH-matrix}, one has
\[
H^0=
  \left[ {\begin{array} {rrr}
   1 & 0 & 0 \\
   0 & 1 & 0 \\
   0 & -1 & 1 \\
  \end{array} } \right],
\qquad
\Phi=
  \left[ {\begin{array} {rrr}
   0 & 1 & 0 \\
   1 & 0 & 0 \\
   0 & 0 & 0 \\
  \end{array} } \right].
\]
For the Pesic-Williams parallel-server example \eqref{eq:HforPW}, one has
\[
H^0=
  \left[ {\begin{array} {rrr}
   2 & 0 & 0 \\
   0 & 1 & 0 \\
   0 & 0 & 1 \\
  \end{array} } \right],
\qquad
\Phi=
  \left[ {\begin{array} {rrr}
   0 & 1/2 & 0 \\
   1 & 0 & 0 \\
   1 & 0 & 0 \\
  \end{array} } \right].
\]
Finally, for the three-station example \eqref{eq:ThirdH}, one has
\[
H^0=\frac{1}{4}
  \left[ {\begin{array} {rrrrrrrr}
   1 &  &  &  &  &  &  &  \\
   -1 & 1 &  &  &  &  &  &  \\
    &  & 1 &  &  &  &  &  \\
    &  &  & 1 &  &  &  &  \\
    &  &  & -1 & 1 &  &  &  \\
    &  &  &  &  & 1 &  &  \\
    &  &  &  &  & -1 & 1 &  \\
    &  &  &  &  &  & -1 & 1 \\
  \end{array} } \right],
\qquad
\Phi=
  \left[ {\begin{array} {rrrrrrrr}
    &  &  &  &  &  & 1 &  \\
    &  &  &  & 1/3 &  &  &  \\
    &  &  & 1/3 &  & 1/3 &  & 1/3 \\
    &  & 1/3 &  &  & 1/3 &  & 1/3 \\
    & 3 &  &  &  &  &  &  \\
    &  & 1/3 & 1/3 &  &  &  & 1/3 \\
   1 &  &  &  &  &  &  &  \\
    &  & 1/3 & 1/3 &  & 1/3 &  &  \\
  \end{array} } \right],
\]
where blank entries are zeros.

In each case $H^0$ is lower triangular with strictly positive diagonal elements and non-positive off-diagonal elements, hence a nonsingular $M$-matrix, while $\Phi \geq 0$ and direct calculation gives $\rho(\Phi) \leq 1$. Thus $\rho(\kappa\Phi)=\kappa\rho(\Phi) \leq \kappa <1$ for $\kappa \in (0,1)$, so $I-\kappa\Phi$ is a nonsingular $M$-matrix as well, and $H$ is completely-$S$ by Lemma \ref{lem:MMproduct}.

\section{Implementation details of our computational method}
\label{app:computation}

This appendix provides details of the implementation of the numerical results presented in Section~\ref{sec:Tests}. Recall that we consider the criss-cross network example, the Pesic-Williams example, and the 3-station network, all of which were introduced in Section~\ref{sec:general} and studied throughout the paper. 

\paragraph{Reference process.}
The data for neural network training are generated using the following reference process:
$$\hat{Z}(t) = X(t) + \mu_{\rm ref} t + HL(t),~ t \geq 0,$$
where \(\mu_{\rm ref}\) is a constant drift-rate vector. We set each coordinate of \(\mu_{\rm ref}\) to -0.5, -1.0, and -0.15 for the criss-cross network example, the Pesic-Williams example, and the 3-station example, respectively. Thus, the reference process is an RBM in the orthant. It differs from the controlled state process in Equation \eqref{eq:BCP1new} only through its drift rate: the controlled process uses $\zeta+R\theta(\cdot)$, whereas the reference process uses the constant vector \( \mu_{\rm ref}\). 

As shown in Table \ref{tab:hparams}, we set the minibatch size to 256; that is, we simulate 256 discretized trajectories of $\hat{Z}(\cdot)$ in parallel using an Euler-Maruyama scheme. Each trajectory is partitioned into 6,000 consecutive segments of duration $T=0.01$, corresponding to $N=64$ time steps in the Euler-Maruyama scheme. We discard the first 1,000 segments of each trajectory as warm-up and retain the remaining 5,000 segments. This yields a pool of 5,000 minibatches, each of size 256, which we cycle through during training as needed. In particular, no new trajectories are generated during training.

\paragraph{Neural network training.}
Each policy is trained on a single H100 GPU. The value function and its gradient are approximated by two feedforward networks, \(V_\eta:\mathbb{R}^m\to\mathbb{R}\) and \(G_\phi:\mathbb{R}^m\to\mathbb{R}^m\), with ELU activations and a linear output layer. 
We add a penalty term to the loss function to encourage the gradient-function estimate $G_{\phi}$ to be nonnegative. This, in turn, ensures that the value-function estimate $V_{\eta}$ is nondecreasing in each of its input variables. Specifically, for each minibatch, we add the average of the penalty
$$\Lambda \sum_{i=1}^m \max\{-G_{\phi}^i(z),0\},$$
where $G_{\phi}^i$ denotes the $i$th component of the gradient estimate and $\Lambda$ is a penalty parameter to be tuned. During training, we gradually increase the parameter that bounds the drift-rate control as a function of the iteration index $t$, as follows:
\[
  b_{\mathrm{eff}}(t)=\min\!\Bigl\{\,b,\; b_0+\frac{t}{40\,\log_2(1+m)\,\mathrm{pace}}\,\Bigr\}.
\]
We start at \(b_0=0\) and reach the target \(b\) at step \(40\,\log_2(1+m)\,\mathrm{pace}\cdot b\) (\(8000\) for \(m=3\), about \(12{,}700\) for \(m=8\)). The learning rate schedule follows a standard cosine decay over training. The remaining hyperparameters are shown in Table~\ref{tab:hparams}; they are shared across the three examples, differing only in the network depth and the number of training updates.

\begin{table}[!ht]
\caption{Hyperparameters for neural network training, policy-evaluation simulations, and the MDP benchmark.}
\label{tab:hparams}
\centering
\small
\begin{tabular}{@{}lccc@{}}
\toprule
 & Criss-cross & Pesic-Williams & Three-station \\
 \midrule
 \multicolumn{4}{@{}l}{\textit{Reference process sampling}}\\
 \quad reference drift (per coordinate) & \(-0.5\) & \(-1.0\) & \(-0.15\) \\
 \quad segment horizon \(T\) & \(0.01\) & \(0.01\) & \(0.01\) \\
 \quad time steps per segment \(N\) & \(64\) & \(64\) & \(64\) \\
 \quad step size \(\Delta t=T/N\) & \(1.5625{\times}10^{-4}\) & \(1.5625{\times}10^{-4}\) & \(1.5625{\times}10^{-4}\) \\
 \quad parallel trajectories (\(=\) minibatch size) & \(256\) & \(256\) & \(256\) \\
 \quad segments per trajectory (warm-up \(+\) retained) & \(1000{+}5000\) & \(1000{+}5000\) & \(1000{+}5000\) \\
 \midrule
\multicolumn{4}{@{}l}{\textit{Neural network architecture}}\\
\quad hidden layers \(\times\) width & \(3\times100\) & \(3\times100\) & \(4\times100\) \\
\quad activation & ELU & ELU & ELU \\
\quad initialization & He & He & He \\
\midrule
\multicolumn{4}{@{}l}{\textit{Neural network optimization}}\\
\quad optimizer & AdamW & AdamW & AdamW \\
\quad base learning rate \(\eta_0\) & \(10^{-3}\) & \(10^{-3}\) & \(10^{-3}\) \\
\quad weight decay & \(10^{-3}\) & \(10^{-3}\) & \(10^{-3}\) \\
\quad gradient-norm clip & \(10\) & \(10\) & \(10\) \\
\quad nonnegative gradient penalty \(\Lambda\) & \(1.0\) & \(1.0\) & \(1.0\) \\
\quad training updates & \(48{,}000\) & \(48{,}000\) & \(128{,}000\) \\
\midrule
\multicolumn{4}{@{}l}{\textit{Training control-bound curriculum}}\\
\quad target bound \(b\) & \(10\) & \(10\) & \(10\) \\
\quad initial bound \(b_0\) & \(0\) & \(0\) & \(0\) \\
\quad pace & \(10\) & \(10\) & \(10\) \\
\midrule
\multicolumn{4}{@{}l}{\textit{Brownian value simulation}}\\
\quad number of paths & \(100{,}000\) & \(100{,}000\) & \(100{,}000\) \\
\quad time horizon & \(3.0\) & \(3.0\) & \(3.0\) \\
\quad step size & \(1.5625{\times}10^{-4}\) & \(1.5625{\times}10^{-4}\) & \(1.5625{\times}10^{-4}\) \\
\quad control bound \(b\) & \(10\) & \(10\) & \(10\) \\
\midrule
\multicolumn{4}{@{}l}{\textit{Queueing-network simulation}}\\
\quad number of paths & \(100{,}000\) & \(100{,}000\) & \(100{,}000\) \\
\quad time horizon (\(=r\cdot 3.0\)) & \(1200\) & \(1200\) & \(1200\) \\
\midrule
\multicolumn{4}{@{}l}{\textit{MDP benchmark (value iteration)}}\\
\quad truncation level \(M\) & \(300\) & \(300\) & --- \\
\quad stopping tolerance \(\varepsilon\) (max norm) & \(10^{-4}\) & \(10^{-4}\) & --- \\
\bottomrule
\end{tabular}
\end{table}

\paragraph{Queueing-network simulations.}
The costs reported in Section~\ref{sec:Tests} are estimated by simulating the continuous-time Markov chain associated with the queueing network under the deployed policy, with the simulation parameters listed in Table~\ref{tab:hparams}. 

\paragraph{Benchmark policies.}
As benchmark policies, we consider the optimal MDP policy whenever the MDP is computationally tractable, as well as a greedy heuristic. We also evaluated several other heuristic policies, but these were dominated by the greedy heuristic. Therefore, for brevity, we present only the greedy heuristic.
\newline
\textit{The greedy policy.} This policy builds on the maximum-pressure policy, cf. \citet{dailin2005maximum}, and uses only the current queue-length vector $q$, the holding-cost vector $h$, and the service-rate vector $\mu$. Crucially, each server makes its decisions locally in a decentralized manner. We next describe how the greedy heuristic makes sequencing, routing, and rejection decisions.

\begin{itemize}
  \item \emph{Sequencing decisions.}  Consider a processing activity $\ell$ that drains buffer $i$ at rate $\mu_\ell$. Under this activity, each completed job either proceeds to a downstream buffer, say buffer $j$, or leaves the network. In the former case, we define the pressure of the activity as $\mu_\ell(h_i q_i - h_j q_j)$. In the latter case, the pressure is $\mu_\ell h_i q_i$. Among processing activities with nonempty input buffers, the server selects the activity with the highest pressure and never idles as long as at least one input buffer is nonempty.
  \item \emph{Routing decisions.} This applies only to the $3$-station example. Under the greedy heuristic, a type-A arrival is routed to buffer 1 if
    \begin{equation*}
        h_1 q_1 < h_3 q_3
    \end{equation*}
    and to buffer 3 otherwise.
  \item \emph{Admission control decisions.} This also applies only to the $3$-station example. We let $r_i$ denote the penalty for rejecting a job arriving to buffer $i$ for $i\in \{1,3,4,6\}$. Thus, we have 
  $$
  r_1=r_3=\frac{c_4}{\lambda_A}, \quad r_4=\frac{c_5}{\lambda_B}, \quad r_6=\frac{c_6}{\lambda_C}.
  $$
  When the system manager has no routing decision to make, as is the case for jobs of types B and C, a job arriving to buffer $i$ is admitted if the buffer pressure is below the rejection penalty, that is, if
    \begin{equation*}
        h_i q_i < r_i, \quad i\in \{4,6\}.
    \end{equation*}
Otherwise, the job is rejected. When both routing and admission control decisions are required, as is the case for type A jobs, the system manager first makes the routing decision using the rule described above. Once the destination buffer has been selected, the admission control decision is made by comparing that buffer’s pressure with the corresponding rejection penalty. Thus, if the job is routed to buffer $j$, it is admitted if
    \begin{equation*}
        h_j q_j < r_j, \quad j\in \{1,3\}.
    \end{equation*}
    and rejected otherwise.

\end{itemize}
\ \newline
\emph{The optimal MDP policy.} Solving the MDP is computationally feasible only for the criss-cross and Pesic-Williams examples, after truncating the state space to $\{0,1,\ldots,M\}^3$. We compute the optimal policy via value iteration: with uniformization constant \(g=\rho+\sum_a\lambda_a+\sum_k\mu_k\), the Bellman update is
\[
  V_{t+1}(q)=\frac1g\Bigl[\,h^\top q+\sum_{a}\lambda_a\,V_t(q{+}e_a)+\min_{u}\sum_k \mu_k\,V_t\bigl(q-\Delta_k(u)\bigr)\Bigr],
\]
where the minimum is over admissible action profiles \(u\) and \(\Delta_k(u)\) is the queue change caused by server \(k\) under \(u\) (zero if \(k\) idles). For the Pesic-Williams example, we also impose the constraint that at most one server may work on each job at any given time. That is, we do not allow two servers to process the same job simultaneously.  Iterations stop when \(\max_{q}\abs{V_{t+1}(q)-V_t(q)}\) falls below the tolerance \(\varepsilon\). We report $V(0)$. The values of \(M\) and \(\varepsilon\) are given in Table~\ref{tab:hparams}.

\end{appendix}
\end{document}